\documentclass[11pt]{article}
\usepackage[margin=2.2cm,centering,nohead]{geometry}
\usepackage{amsmath, amsthm, amssymb, mathdots}
\usepackage{graphicx}
\usepackage[small,it]{caption}
\usepackage{tikz}
\usetikzlibrary{calc,backgrounds,patterns}
\usepackage[T1]{fontenc}
\usepackage[sc]{mathpazo}
\usepackage{parskip}
\usepackage{todonotes}

\definecolor{darkblue}{rgb}{0, 0, .4}
\usepackage[bookmarks]{hyperref}
\hypersetup{
    colorlinks=true,
    linkcolor=darkblue,
    anchorcolor=darkblue,
    citecolor=darkblue,
    urlcolor=darkblue,
    pdfpagemode=UseThumbs,
    pdftitle={Characterising inflations of monotone grid classes of permutations},
    pdfsubject={Combinatorics},
    pdfauthor={M.H. Albert, A. Atminas and R. Brignall},
    pdfkeywords={permutation,pattern,simple permutation,grid classes}
}

\title{Characterising inflations of monotone grid classes of permutations}
\author{
Michael Albert\thanks{Email addresses: \texttt{malbert@cs.otago.ac.nz, aistis.atminas@gmail.com, rbrignall@gmail.com}.}\\[10pt]
\small Department of Computer Science\\
\small University of Otago\\
\small Dunedin\\
\small New Zealand
\and
Aistis Atminas\footnotemark[1]\\[10pt]
\small Department of Mathematics\\
\small London School of Economics\\
\small London, WC2A 2AE\\
\small United Kingdom\and
Robert Brignall\footnotemark[1]\\[10pt]
\small School of Mathematics and Statistics\\
\small The Open University\\
\small Milton Keynes, MK7 6AA\\
\small United Kingdom}

\newcommand{\C}{\mathcal{C}}
\newcommand{\D}{\mathcal{D}}

\newcommand{\R}{\mathcal{R}}
\renewcommand{\S}{\mathcal{S}}
\newcommand{\rect}{\operatorname{rect}}

\newcommand{\Av}{\operatorname{Av}}
\newcommand{\Si}{\operatorname{Si}}

\begingroup
    \makeatletter
    \@for\theoremstyle:=definition,remark,plain\do{%
        \expandafter\g@addto@macro\csname th@\theoremstyle\endcsname{%
            \addtolength\thm@preskip\parskip
            }%
        }
\endgroup

\theoremstyle{plain}
\newtheorem{theorem}{Theorem}[section]
\newtheorem{question}[theorem]{Question}
\newtheorem{lemma}[theorem]{Lemma}

\newtheorem{proposition}[theorem]{Proposition}
\newtheorem{observation}[theorem]{Observation}

\usepackage{paralist}

\usepackage{ifthen}
\usetikzlibrary{calc}

\newcommand\absdot[2]{
	\node at #1 {\normalsize $\bullet$};
	\node at #1 [below] {$#2$};
}

\newcommand\absdothollow[2]{
	\node at #1 {\normalsize \textcolor{white}{$\bullet$}};
	\node at #1 {\normalsize $\circ$};
	\node at #1 [below] {$#2$};
}
\newcommand{\plotperm}[1]{
	\foreach \j [count=\i] in {#1} {
		\absdot{(\i,\j)}{};
	};
}
\newcommand{\plotpermbox}[4]{
	\draw [darkgray, very thick, line cap=round]
		({#1-0.5}, {#2-0.5}) rectangle ({#3+0.5}, {#4+0.5});
}

\newcommand{\plotpermborder}[1]{
	\plotperm{#1};
	\foreach \i [count=\n] in {#1} {};
	\plotpermbox{1}{1}{\n}{\n};
}

\newcommand{\plotpinsequence}[1]{
	\absdot{(0,0)}{};
	\edef\n{0}
	\edef\s{0}
	\edef\e{0}
	\edef\w{0}
	\edef\x{0}
	\edef\y{0}
	\foreach \pin [remember=\pin as \oldpin (initially 1), count=\i] in {#1} {
		\ifthenelse{\pin=1 \OR \pin=2}{
			\ifthenelse{\oldpin=3}{
				\xdef\x{\number\numexpr\e-1}
			}{
				\xdef\x{\number\numexpr\w+1}
			}
			\ifnum\i=1 
				\pgfmathparse{\e+1}
 				\xdef\e{\pgfmathresult}
			\fi	
		}{ 
			\ifthenelse{\oldpin=1}{
				\xdef\y{\number\numexpr\n-1}
			}{
				\xdef\y{\number\numexpr\s+1}
			}
			\ifnum\i=1 
				\pgfmathparse{\s-1}
 				\xdef\s{\pgfmathresult}
			\fi	
		}
		\ifnum\pin=1 
			\pgfmathparse{\n+2}
 			\xdef\n{\pgfmathresult}		
			\absdot{(\x,\n)}{};
			\ifnum\i>1
				\draw (\x,\n) -- (\x,\y-0.5);
			\else
				\draw[gray,very thick] (-0.5,-0.5) rectangle (\x+0.5,\n+0.5);
			\fi
		\fi
		\ifnum\pin=2 
			\pgfmathparse{\s-2}
 			\xdef\s{\pgfmathresult}
			\absdot{(\x,\s)}{};
			\ifnum\i>1
				\draw (\x,\s) -- (\x,\y+0.5);
			\else
				\draw[gray,very thick] (-0.5,0.5) rectangle (\x+0.5,\s-0.5);
			\fi
		\fi
		\ifnum\pin=3 
			\pgfmathparse{\e+2}
 			\xdef\e{\pgfmathresult}
			\absdot{(\e,\y)}{};
			\ifnum\i>1
				\draw (\e,\y) -- (\x-0.5,\y);
			\else
				\draw[gray,very thick] (-0.5,+0.5) rectangle (\e+0.5,\y-0.5);
			\fi
		\fi
		\ifnum\pin=4 
			\pgfmathparse{\w-2}
 			\xdef\w{\pgfmathresult}
			\absdot{(\w,\y)}{};
			\ifnum\i>1
				\draw (\w,\y) -- (\x+0.5,\y);
			\else
				\draw[gray,very thick] (0.5,0.5) rectangle (\w-0.5,\y-0.5);

			\fi
		\fi		
	};
}

\begin{document}
\maketitle
\begin{abstract}
We characterise those permutation classes whose simple permutations are monotone griddable. This characterisation is obtained by identifying a set of nine substructures, at least one of which must occur in any simple permutation containing a long sum of 21s.
\end{abstract}

\section{Introduction}

A common route to understanding the structure of a permutation class (and hence, e.g.\ complete its enumeration) is via its simple permutations, as their structure can be considerably easier to characterise than the entire class. Albert, Atkinson, Homberger and Pantone~\cite{albert:deflatability} introduced the notion of \emph{deflatability} to study this phenomenon: that is, the property that the simples in a given permutation class $\C$ actually belong to a proper subclass $\D\subsetneq\C$. See also Vatter's recent survey~\cite{vatter:survey}.

One general case of deflatability is where the set of simple permutations of a class is finite. Such classes are well-quasi-ordered, finitely based, and have algebraic generating functions~\cite{albert:simple-permutat:}, and via a Ramsey-type result for simple permutations~\cite{brignall:decomposing-sim:}, it is decidable when a permutation class has this property~\cite{brignall:simple-permutat:b}.

In this paper, we look beyond classes with finitely many simples to those whose simples are `monotone griddable', and prove the following characterisation. We postpone formal definitions until later, but see Figure~\ref{fig-griddables} for examples of the structures mentioned.

\begin{theorem}\label{thm:griddable}
The simple permutations in a class $\C$ are monotone griddable if and only if $\C$ does not contain the following structures, or their symmetries:
\begin{compactitem}
\item arbitrarily long parallel sawtooth alternations,
\item arbitrarily long sliced wedge sawtooth alternations,
\item proper pin sequences with arbitrarily many turns, and
\item spiral proper pin sequences with arbitrarily many extensions.
\end{compactitem}
\end{theorem}

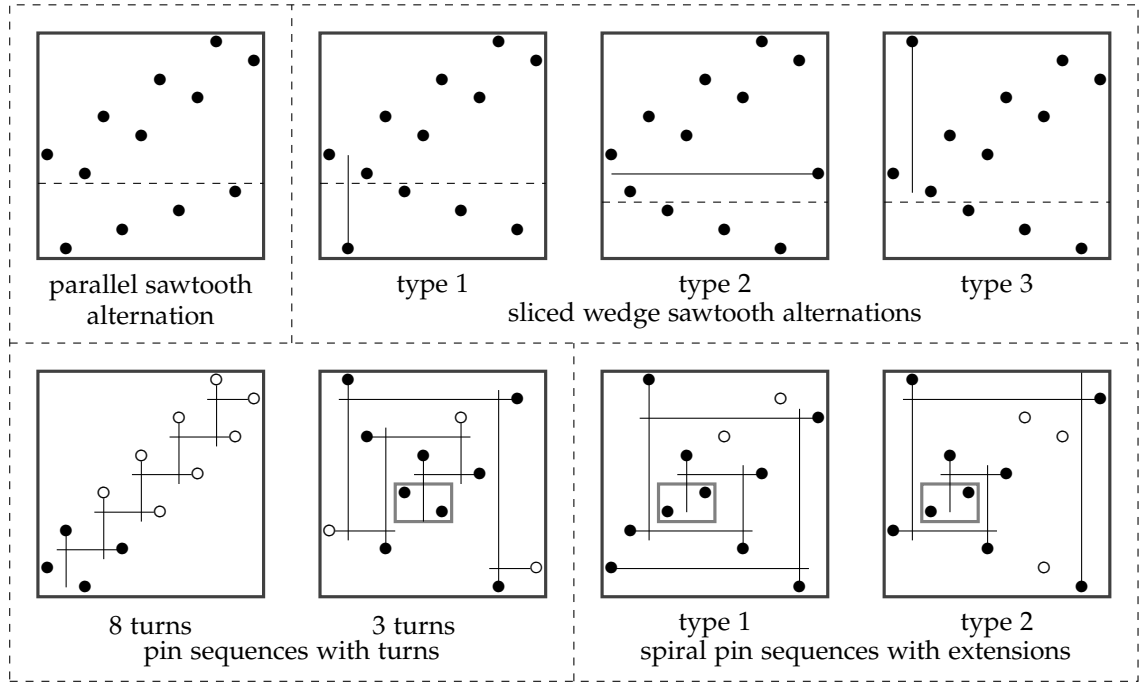
\begin{figure}
\centering
\begin{tikzpicture}[scale=0.25]
\begin{scope}[shift={(0,0)}]
\plotpermborder{6,1,5,8,2,7,10,3,9,12,4,11}
\draw[dashed] (0.5,4.5) -- (12.5,4.5);
\node at (6.5,-1) {\small parallel sawtooth};
\node at (6.5,-2.5) {\small alternation};
\end{scope}
\begin{scope}[shift={(15,0)}]
\plotpermborder{6,1,5,8,4,7,10,3,9,12,2,11}
\draw[dashed] (0.5,4.5) -- (12.5,4.5);
\draw (2,6) -- (2,1);
\node at (6.5,-1) {\small type 1};
\end{scope}
\begin{scope}[shift={(30,0)}]
\plotpermborder{6,4,8,3,7,10,2,9,12,1,11,5}
\draw[dashed] (0.5,3.5) -- (12.5,3.5);
\draw (1,5) -- (12,5);
\node at (6.5,-1) {\small type 2};
\node at (6.5,-2.5) {\small sliced wedge sawtooth alternations};
\end{scope}
\begin{scope}[shift={(45,0)}]
\plotpermborder{5,12,4,7,3,6,9,2,8,11,1,10}
\draw[dashed] (0.5,3.5) -- (12.5,3.5);
\draw (2,4) -- (2,12);
\node at (6.5,-1) {\small type 3};
\end{scope}
\begin{scope}[shift={(0,-18)}]
\plotpermborder{2,4,1,6,3,8,5,10,7,12,9,11}
\draw (2,1) -- (2,4);
\draw (1.5,3) -- (5,3);
\foreach \x/\y/\z/\w in {4/2.5/4/6,3.5/5/7/5,6/4.5/6/8,5.5/7/9/7, 8/6.5/8/10,7.5/9/11/9,10/8.5/10/12,9.5/11/12/11} {
	\draw (\x,\y) -- (\z,\w);
	\absdothollow{(\z,\w)}{};
}
\node at (6.5,-1) {\small 8 turns};
\end{scope}
\begin{scope}[shift={(20,-12)}]
\plotpermbox{-4}{-5}{7}{6}
\plotpinsequence{3,1,3,1,4,2,4,1,3,2}
\draw (7,-4) -- +(-2.5,0);
\absdothollow{(7,-4)}{};
\absdothollow{(-4,-2)}{};
\absdothollow{(3,4)}{};
\node at (0.5,-7) {\small 3 turns};
\end{scope}
\begin{scope}[shift={(30,-18)}]
\draw[very thick,gray] (3.5,4.5) rectangle (6.5,6.5);
\foreach \x/\y/\z/\w in {5/5/5/8,4.5/7/9/7,8/7.5/8/3,8.5/4/2/4, 3/3.5/3/12,2.5/10/12/10,11/10.5/11/1,11.5/2/1/2} {
	\draw (\x,\y) -- (\z,\w);
}
\plotpermborder{2,4,12,5,8,6,9,3,7,11,1,10}
\absdothollow{(7,9)}{};
\absdothollow{(10,11)}{};
\node at (6.5,-1) {\small type 1};
\end{scope}
\begin{scope}[shift={(45,-18)}]
\draw[very thick,gray] (2.5,4.5) rectangle (5.5,6.5);
\foreach \x/\y/\z/\w in {4/5/4/8,3.5/7/7/7,6/7.5/6/3,6.5/4/1/4, 2/3.5/2/12,1.5/11/12/11,11/12.5/11/1} {
	\draw (\x,\y) -- (\z,\w);
}
\plotpermborder{4,12,5,8,6,3,7,10,2,9,1,11}
\absdothollow{(8,10)}{};
\absdothollow{(9,2)}{};
\absdothollow{(10,9)}{};
\node at (6.5,-1) {\small type 2};
\end{scope}
\draw[dashed] (-1,-22) rectangle (59,14);
\draw[dashed] (-1,-4) -- (59,-4);
\draw[dashed] (29,-22) -- (29,-4);
\draw[dashed] (14,-4) -- (14,14);
\node at (14,-20.5) {\small pin sequences with turns};
\node at (44,-20.5) {\small spiral pin sequences with extensions};
\end{tikzpicture}
\caption{Examples of the permutations characterising the griddability of simples in Theorem~\ref{thm:griddable}.}
\label{fig-griddables}
\end{figure}

In general, classes whose simple permutations are monotone griddable do not immediately possess the range of properties that classes with only finitely many simples do. Indeed, few general properties are known even for classes that are themselves wholly monotone griddable, but this has not diminished the efficacy of the following characterisation for the structural understanding and enumeration of many classes (see, for example~\cite{aab:three-enumerations:}).

A \emph{sum of $k$ copies of 21} is the permutation $21\,43\,65\cdots (2k)(2k-1)$, written in one line notation. We will often abbreviate this to $\oplus_k 21$. Similarly, a \emph{skew sum of $k$ copies of 12} is the permutation $\ominus_k12 = (2k-1)(2k)\cdots34\,12$.

\begin{theorem}[Huczynska and Vatter~\cite{huczynska:grid-classes-an:}]\label{thm-hv-griddable}
A class $\C$ is monotone griddable if and only if it does not admit arbitrarily long sums of 21 or skew sums of 12. That is, for some $k$ neither $\oplus_k 21$ nor $\ominus_k 12$ belong to $\C$.
\end{theorem}

Aside from the structural information that it provides in its own right, the reason that this simple-to-check characterisation has proved so useful is that the classes to which it has been applied typically in fact possess the stronger property of being \emph{geometrically griddable}. Such classes are well-quasi-ordered, finitely based and have rational generating functions~\cite{albert:geometric-grid-:}. As we have no direct characterisation for a class to be geometrically griddable, the above theorem (which certainly provides a necessary condition) has often provided enough of a `hook' to solve the task at hand.

It is our hope that Theorem~\ref{thm:griddable} can provide a similar `hook' to ease the study of classes whose simple permutations are geometrically griddable. Any such class is known to be well-quasi-ordered, finitely based, and strongly algebraic (meaning that it and every subclass have algebraic generating functions), see Albert, Ru\v{s}kuc and Vatter~\cite{albert:inflations-of-g:}. Furthermore, every class with growth rate less than $\kappa\approx 2.20557$, is of this form~\cite{albert:inflations-of-g:}. For instances of practical enumeration tasks that have exploited the geometric griddability of the simple permutations, see Albert, Atkinson and Vatter~\cite{albert:inflations-of-g:2x4}, and Pantone~\cite{pantone:the-enumeration}.

Our characterisation in Theorem~\ref{thm:griddable} relies on the following auxiliary result, which guarantees the existence of certain types of structure in simple permutations that contain a long sum of 21s. 

\begin{theorem}\label{thm-main}
There exists a function $f(n)$ such that every simple permutation that contains a sum of $f(n)$ copies of 21 must contain a parallel or wedge sawtooth alternation of length $3n$ or an increasing oscillation of length $n$.
\end{theorem}

See Figure~\ref{fig-sawtooth} for examples of the three types of unavoidable structure. Note that wedge sawtooth alternations are not necessarily simple, so the existence of wedge sawtooth alternations in a permutation class does not guarantee that the simple permutations are not monotone griddable, but Theorem~\ref{thm-main} nevertheless provides a sufficient condition.

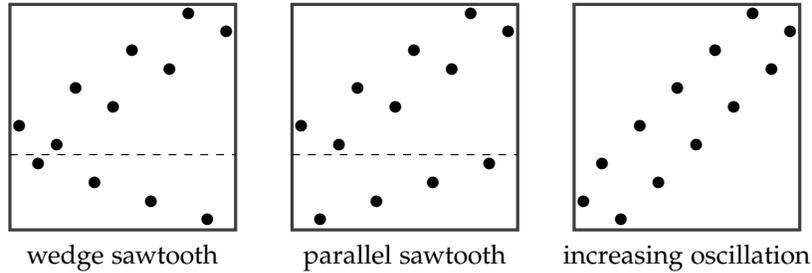
\begin{figure}
\centering
\begin{tikzpicture}[scale=0.25]
\begin{scope}[shift={(0,0)}]
\plotpermborder{6,4,5,8,3,7,10,2,9,12,1,11}
\draw[dashed] (0.5,4.5) -- (12.5,4.5);
\node at (6.5,-1) {\small wedge sawtooth};
\end{scope}
\begin{scope}[shift={(15,0)}]
\plotpermborder{6,1,5,8,2,7,10,3,9,12,4,11}
\draw[dashed] (0.5,4.5) -- (12.5,4.5);
\node at (6.5,-1) {\small parallel sawtooth};
\end{scope}
\begin{scope}[shift={(30,0)}]
\plotpermborder{2,4,1,6,3,8,5,10,7,12,9,11}
\node at (6.5,-1) {\small increasing oscillation};
\end{scope}
\end{tikzpicture}
\caption{Examples of the two types of sawtooth alternation and an increasing oscillation in the statement of Theorem~\ref{thm-main}.}
\label{fig-sawtooth}
\end{figure}

The rest of this paper is organised as follows. In Section~\ref{preliminaries} we introduce basic notions and define the structures mentioned in the above results. In Section~\ref{sec-main} we give the proof of Theorem~\ref{thm-main}, and in Section~\ref{sec-gridding} we complete the proof of Theorem~\ref{thm:griddable}. In the final section, we discuss some future directions for this work.

%
%
%
%
%
%
%
%
%
\section{Preliminaries}\label{preliminaries}

For definitions common to the wider study of permutation patterns, we refer the reader to Bevan's introduction~\cite{bevan2015defs}. For a broader background to the study of permutation classes, see Vatter's excellent survey~\cite{vatter:survey} in the \emph{Handbook of enumerative combinatorics}.

\subsection{Geometry, simplicity and gridding}

Critical to our work is the ability to visualise permutations and parts of permutations in the plane. The \emph{plot} of a permutation $\pi$ of length $n$ is the set of coordinates (or \emph{points}) $(i,\pi(i))$ for $i=1,\dots,n$. In a slight abuse of notation, we will rarely distinguish between a permutation and its plot. 

This exposes an important collection of symmetries that are available -- specifically the dihedral group generated by reflections in a vertical, horizontal, or diagonal axis. It is to these symmetries we refer when we make an appeal `by symmetry'. In particular to prove Theorem \ref{thm:griddable} it suffices to show that if the simple permutations of a class $\C$ contain arbitrarily long sums of 21, then configurations of one of the specified types must occur.

Given points $p_1,\dots,p_k$ in the plane (typically belonging to the plot of a permutation), denote by $\rect(p_1,\dots,p_k)$ the smallest axes-parallel rectangle that contains them. We call $\rect(p_1,\dots,p_k)$ the \emph{rectangular hull} of $p_1,\dots,p_k$. 

Let $\R$ be any axes-parallel rectangle in the plot of a permutation $\pi$.   The rectangle $\R$ divides the plot of $\pi$ into nine regions, and we identify the four `corners' as \emph{NE}, \emph{NW}, \emph{SE} and \emph{SW} of $\pi$ relative to $\R$, as illustrated in the following diagram.

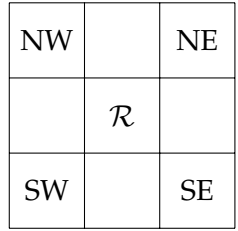
\begin{center}
\begin{tikzpicture}
\draw (0,0) grid (3,3);
\node at (1.5,1.5) {$\R$};
\node at (2.5,2.5) {NE};
\node at (0.5,2.5) {NW};
\node at (2.5,0.5) {SE};
\node at (0.5,0.5) {SW};
\end{tikzpicture}
\end{center}

For the rectangle $\R$ itself, denote by $\pi|_{\R}$ the permutation that is order isomorphic to the points of $\pi$ contained in $\R$. 

Any point (or collection of points) that lies in one of the four unlabelled regions in the above picture is said to \emph{slice} the rectangle $\R$. Put formally, if $\R=[a,b]\times [c,d]$ is a rectangle, then the point $(x,y)$ \emph{slices} $\R$ \emph{vertically} if $x\in (a,b)$ and $y\not\in [c,d]$, and \emph{horizontally} if $x\not\in[a,b]$ and $y\in(c,d)$. By extension, we say that a point \emph{slices} a collection of points in the plane if it slices their rectangular hull.

An \emph{interval} of a permutation $\pi$ is a (nonempty) set of points $\{(i,\pi(i)): i\in I\}$ for some set of indices $I$, such that both $I$ and $\pi(I)=\{\pi(i):i\in I\}$ form contiguous sets of natural numbers.

We can easily identify an interval geometrically by noting that $\rect((i,\pi(i)):i\in I)$ cannot be sliced, and must contain only points corresponding to indices from $I$. Equivalently, the (nonempty) set of points of $\pi$ belonging to an unsliced axes-parallel rectangle $\R$ form an interval. 

Trivially, every singleton of $\pi$ and the whole of $\pi$ form intervals. If there are no other intervals and $\pi \neq 1$, then $\pi$ is said to be \emph{simple}. 

Given a permutation $\sigma$ of length $n$, and permutations $\pi_1,\dots,\pi_n$, the \emph{inflation} of $\sigma$ by $\pi_1,\dots, \pi_n$ is the permutation obtained by replacing each entry $\sigma(i)$ by a sequence of points forming an interval order isomorphic to $\pi_i$, and with the intervals in the same relative ordering as the corresponding points of $\sigma$. This permutation is commonly denoted by $\sigma[\pi_1,\dots,\pi_n]$.

The reverse process to inflation (i.e., decomposing a permutation into intervals) forms the basis for the \emph{substitution decomposition}, the essence of which is captured in the following result.

\begin{proposition}[Albert and Atkinson~\cite{albert:simple-permutat:}]\label{prop-mod-decomp}
Every permutation $\pi$ is expressible as the inflation of a unique simple permutation $\sigma$. Furthermore, if $|\sigma|\geq 4$, then in the expression
\[\pi = \sigma[\pi_1,\dots,\pi_n],\]
 the intervals $\pi_1,\dots\pi_n$ are also unique.
\end{proposition}

For a permutation which is the inflation of a simple permutation of length 2 (i.e., $\sigma=12$ or $21$), we do not have the same guarantee of uniqueness of the intervals (although there are methods to recover this if needed). If $\pi=12[\pi_1,\pi_2]$ for some permutations $\pi_1$ and $\pi_2$, then we also write $\pi=\pi_1\oplus \pi_2$, and say that $\pi$ is \emph{sum decomposable}. Any permutation that is not sum decomposable is \emph{sum indecomposable}. Similarly, if $\pi=21[\pi_1,\pi_2]$ we write $\pi=\pi_1\ominus \pi_2$ and say that $\pi$ is \emph{skew decomposable}, otherwise $\pi$ is \emph{skew indecomposable}. Finally, the case where $\pi$ is both sum and skew indecomposable corresponds to the case in Proposition~\ref{prop-mod-decomp} where the unique simple permutation has length at least 4 (as there are no simple permutations of length 3).

For completeness, we now briefly introduce the notion of griddability. However, we do not actually require this definition in our work (the characterisation provided by Theorem~\ref{thm-hv-griddable} suffices). A class $\C$ is \emph{monotone griddable} if there exist integers $h$ and $v$ such that for every permutation $\pi \in \C$, we may divide the plot of $\pi$ into cells using at most $h$ horizontal and $v$ vertical lines, in such a way as the points in each cell form a monotone increasing or decreasing sequence (or the cell is empty).

\subsection{Pin sequences}

Following~\cite{brignall:decomposing-sim:}, a \emph{pin sequence} is a sequence of points $p_1, p_2,\dots$ in the plot of $\pi$ such that for each $i\ge 3$, $p_i$ slices $\rect(p_1,\dots,p_{i-1})$. Each pin $p_i$ for $i\ge 3$ has a direction -- one of \emph{left}, \emph{right}, \emph{up} or \emph{down} -- based on its position relative to the rectangle that it slices.  By convention, the pins $p_1$ and $p_2$ will be regarded as having no direction.

A {\it proper pin sequence\/} is one that satisfies two additional conditions:
\begin{itemize}
\item \emph{Maximality}: each pin must be maximal in its direction.  For example, if $p_i$ is a right pin, then there are no points further to the right of $p_i$ that slice $\rect(p_1,\dots,p_{i-1})$.
\item \emph{Separation}: $p_{i+1}$ must \emph{separate} $p_i$ from $\{p_1,\dots,p_{i-1}\}$.  In other words, $p_{i+1}$ must lie horizontally or vertically between $\rect(p_1,\dots,p_{i-1})$ and $p_i$.
\end{itemize} 

As all pin sequences required in the sequel will be proper, for brevity we will sometimes use the term `pin sequence' to mean a proper pin sequence. We now recall some basic properties of (proper) pin sequences.

\begin{proposition}[Brignall, Huczynska and Vatter~\cite{brignall:decomposing-sim:}]\label{pinseq1}
In a (proper) pin sequence $p_1,\dots,p_m$,
\begin{enumerate}[(a)]
\item $p_{i+1}$ cannot lie in the same or opposite direction as $p_i$ (for all $i\ge 3$);
\item $p_{i+1}$ does not slice $\rect(p_1,\dots,p_{i-1})$;
\item $p_i$ and $p_{i+1}$ are separated either by $p_{i-1}$ or by each of $p_1, \dots, p_{i-2}$; and
\item one of the sets of points $\{p_1,\dots,p_m\}$, $\{p_1,\dots,p_m\}\setminus\{p_1\}$, or $\{p_1,\dots,p_m\}\setminus\{p_2\}$ is order isomorphic to a simple permutation.
\end{enumerate}
\end{proposition}

See the lower-left part of Figure~\ref{fig-griddables} for examples of pin sequences. 

The following result is critical to what will follow later. A pin sequence $p_1,\dots,p_m$ in a permutation $\pi$ is said to be \emph{right reaching} if $p_m$ is the rightmost point of $\pi$. 

\begin{lemma}[Brignall, Huczynska and Vatter~\cite{brignall:decomposing-sim:}]\label{right-reaching}
For every simple permutation $\pi$ and pair of points $p_1$ and $p_2$ (unless, trivially, $p_1$ is the right-most point of $\pi$), there is a (proper) right-reaching pin sequence beginning with $p_1$ and $p_2$.
\end{lemma}

We now introduce some new terminology relating to pins that we will require for our characterisation. Let $p_1,\dots,p_m$ be a pin sequence in a permutation $\pi$. We say that a pin $p_i$ \emph{turns} if the direction of $p_i$ is the same as the direction of $p_{i-2}$. The significance of this concept is in the following observation.

\begin{lemma}\label{lem-turns}
If a pin sequence $p_1,\dots,p_m$ contains $3(p+q)$ turns, then the permutation corresponding to $p_1,\dots,p_m$ contains $\oplus_p 21$ or $\ominus_q 12$.\end{lemma}

\begin{proof}
For a pin sequence $p_1,\dots,p_m$, let $p$ denote the length of the longest sum of 21s and $q$ the length of the longest skew sum of 12s  that the permutation corresponding to the sequence contains. Let $\ell(p_1,\dots,p_m) = p+q$. 

We will prove by induction on $k$ the following statement: every pin sequence $p_1,\dots,p_m$ containing $3k$ turns satisfies $\ell(p_1,\dots,p_m) \ge k$. The lemma will follow directly.

The base case $k=0$ is trivially true, so let $p_1,\dots,p_m$ be a pin sequence with $3(k+1)$ turns. Let $p_j$ be the latest pin in the sequence which forms a turn, noting that $j\geq 3(k+1)+ 3\geq 6$ (since the first three pins cannot be turns). By symmetry, we may assume that $p_j$ is a right pin (and hence so is $p_{j-2}$) and $p_{j-1}$ an up pin. Note also that $p_1,\dots,p_{j-3}$ contains at least $3(k+1)-3 = 3k$ turns, so by induction we know that $\ell(p_1,\dots, p_{j-3})\geq k$. We now have the following situation:
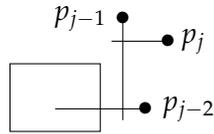
\begin{center}
\begin{tikzpicture}[scale=0.3]
\draw (0,0) rectangle (4,3);
\foreach \x/\y/\z/\w in {2/1/6/1,5/0.5/5/5,4.5/4/7/4} {
	\draw (\x,\y) -- (\z,\w);
	\absdot{(\z,\w)}{};
}
	\node at (7.9,1) {$p_{j-2}$};
	\node at (3.1,5) {$p_{j-1}$};
	\node at (8.1,4) {$p_{j}$};
\end{tikzpicture}
\end{center}
By inspection, we see that the pair of points $p_{j-1},p_j$ forms a copy of 21 that is NE of $\rect(p_1,\dots,p_{j-3})$, and so we may add this copy of 21 to the longest sum of 21s that can be found in $\rect(p_1,\dots,p_{j-3})$. Thus we conclude $\ell(p_1\dots,p_m) \geq \ell(p_1,\dots,p_j) \geq k+1$.
\end{proof}

\subsection{The permutations in Theorems~\ref{thm:griddable} and~\ref{thm-main}}\label{subsec-structures}

\paragraph{Sawtooth alternations} A \emph{sawtooth alternation of length $3n$} is a permutation on $3n$ points that contains a sum of $n$ copies of 21 placed alongside (horizontally or vertically) a monotone sequence of $n$ points, in such a way as each copy of 21 is sliced by a single entry from the monotone sequence -- see the first two illustrations in Figure~\ref{fig-sawtooth}.\footnote{Note that for convenience we are only using the term `sawtooth alternation' in respect of sums of 21. By symmetry there are analogous structures that involve skew sums of 12, but we do not need to give these special names.}

We divide the family of sawtooth alternations into two types: a \emph{parallel sawtooth alternation} is one in which the monotone sequence is increasing, while a \emph{wedge sawtooth alternation} is one in which the  monotone sequence is decreasing. See Figure~\ref{fig-sawtooth} (on page~\pageref{fig-sawtooth}).

It is easy to verify that for $n\geq 2$, the parallel sawtooth alternations of length $3n$ are simple. On the other hand, no wedge sawtooth alternation is simple. This fact underpins the extra work that is required in order to get from Theorem~\ref{thm-main} to Theorem~\ref{thm:griddable}. 

To recover simplicity in wedge sawtooth alternations, consider the sawtooth alternation shown on the left of Figure~\ref{fig-sawtooth}. The leftmost three points of this permutation forms an interval that is order isomorphic to 312. In order to break this (and every other) interval, we form \emph{sliced} wedge sawtooth alternations in one of three ways: pull the `1' of this 312 below all the other points of the monotone sequence (\emph{type 1}), pull the `2' to the right of all other points of the permutation (\emph{type 2}), or replace the `1' with a new maximal element in the permutation (\emph{type 3}). See the top-right portion of Figure~\ref{fig-griddables} (on page~\pageref{fig-griddables}). For wedge sawtooth alternations that are oriented differently, we make the analogous definitions by appealing to symmetry.

\paragraph{Proper pin sequences with turns} As defined in the previous subsection, a \emph{turn} in a pin sequence is a pin $p_i$ that has the same direction as $p_{i-2}$. By Lemma~\ref{lem-turns}, a pin sequence that contains a lot of turns also contains a long sum of 21s or skew sum of 12s (or both). 
See the bottom left portion of Figure~\ref{fig-griddables}, where the pins that are turns have been marked with hollow points. By Proposition~\ref{pinseq1}(d), pin sequences with turns either correspond to simple permutations, or they correspond to permutations for which we may remove one point to recover a simple permutation. 

\paragraph{Increasing oscillations} An \emph{increasing oscillation} of length $n$ is a permutation on $n$ points formed by a pin sequence that starts from a copy of 21, and then is entirely made up of right and up pins. (That is, for every $i\geq 5$ the pin is a turn.) There are two increasing oscillations of each length, which may be obtained from one another by symmetry. See the rightmost illustration in Figure~\ref{fig-sawtooth} (on page~\pageref{fig-sawtooth}).

\paragraph{Extended spiral pin sequences} A pin sequence $p_1,\dots,p_m$ that contains no turns must either follow the repeating pattern of directions `left, up, right, down' or `left, down, right, up'. We call both of these pin sequences \emph{spirals}. 

Unlike pin sequences with many turns, spiral pin sequences do not contain long sums of 21 or skew sums of 12 (indeed, spiral pin sequences are contained in the class of skew-merged permutations, $\Av(3412,2143)$). To recover long sums of 21, we add points in specific locations that we call \emph{extensions}. We will consider two types. For ease of explanation, we assume that $p_{i-1}$ is an up pin and $p_{i}$ a right pin; all other cases follow by symmetry.

\begin{description}
\item[Type 1:] An additional point $q$ is a \emph{type 1 extension of $p_i$} if $\rect(p_i,q)$ is sliced by $p_{i-1}$ and/or $p_{i+1}$, and by no other points, and the points $p_i,q$ form a copy of 21.
\item[Type 2:] Three additional points $q,r,s$ that are placed relative to $p_i$ form a \emph{type 2 extension of $p_i$} if:
\begin{enumerate}[(i)] 
\item $s$ lies either so that $p_{i-1},s$ forms a copy of 21 and the only pin slicing $\rect(p_{i-1},s)$ is $p_i$, or so that $s,p_{i+1}$ forms a copy of 21 and the only pin slicing $\rect(p_{i+1},s)$ is $p_{i+2}$;
\item $q$ and $r$ form a copy of 21 that is SW of $p_i$, NE of $\rect(p_1,\dots,p_{i-2})$ and is sliced only by $s$; and
\item $p_{i+1}$ separates $p_i$ from $q$ and $r$, and it is the only pin other than $s$ that slices $\rect(p_i,q,r)$. 
\end{enumerate}
\end{description}

\begin{figure}
\centering
\begin{tabular}{ccc}
\tikzstyle{background grid}=[draw, black!20,step=0.5]
\begin{tikzpicture}[scale=0.5,show background grid]
\plotpinsequence{3,1,3,2,4,1,3,2,4}
\foreach \x/\y in {4/3,4/4,5/4} {
\absdothollow{(\x+0.5,\y+0.5)}{};
}
\end{tikzpicture}
&\rule{10pt}{0pt}&
\tikzstyle{background grid}=[draw, black!20,step=0.5]
\begin{tikzpicture}[scale=0.5,show background grid]
\plotpinsequence{3,1,3,2,4,1,3,2,4}
\absdothollow{(4.25,2.75)}{}
\absdothollow{(4.75,2.25)}{}
\absdothollow{(-0.5,2.5)}{}
\absdothollow{(4.5,-3.5)}{}
\end{tikzpicture}\\
Type 1&&Type 2
\end{tabular}
\caption{Forming extended spirals. For Type 1, any one of the three hollow points may be added. For Type 2, the copy of 21 is added, together with one of the two slicing points.}
\label{fig:spiral-extensions}
\end{figure}
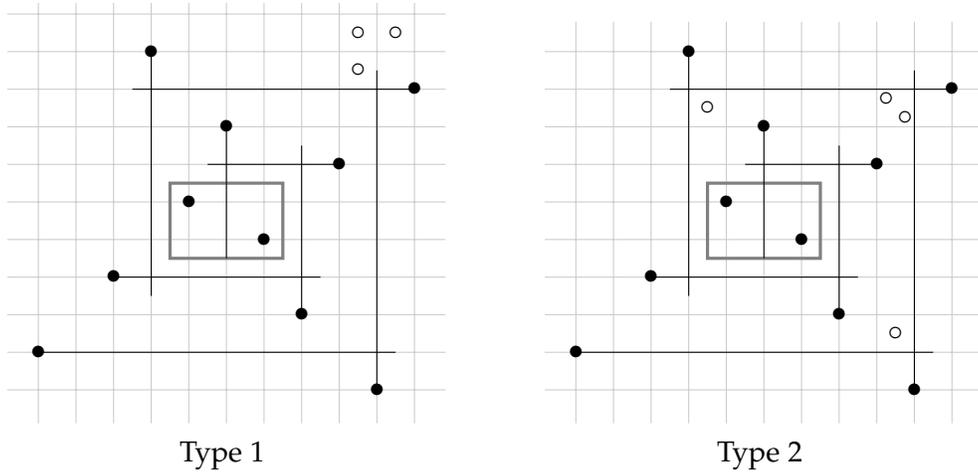

A \emph{spiral pin sequence with $k$ extensions} is a spiral pin sequence for which there exists $k$ \emph{distinct} pins to which extensions of either type have been added.  See Figure~\ref{fig:spiral-extensions} for illustrations of the possible extensions to spirals, and the lower right portion of Figure~\ref{fig-griddables} for two examples of a spiral pin sequence with several extensions. 

We observe that any spiral pin sequence with $k$ extensions is simple: starting from the fact that a spiral pin sequence is itself simple, the only possible intervals that could be created when the extensions are added can contain at most one point of the original pin sequence. Furthermore, every point belonging to a Type 1 extension is separated from any other point by at least one point belonging to the spiral, which prevents these points from being contained in a proper interval. For a Type 2 extension, the only exception to this is that the two points $q$ and $r$ (forming the copy of 21) are separated by the third point $s$, but $\rect(q,r,s)$ is sliced by at least two points from the original spiral.

\begin{lemma}
A spiral pin sequence with $2k$ extensions contains either $\oplus_k 21$ or $\ominus_k 12$.
\end{lemma}

\begin{proof}
For any of the three possible type 1 extensions $q$ to the pin $p_i$ in Figure~\ref{fig:spiral-extensions}, we see that $p_i,q$ forms a copy of 21 that is NE of $\rect(p_1,\dots,p_{i-2})$. Similarly, either of the type 2 extensions also provides a copy of 21 that is NE of the same rectangle, $\rect(p_1,\dots,p_{i-2})$. Similar arguments apply by symmetry to the other corners NW, SE and SW. Thus, from $2k$ extensions, by symmetry we can find $k$ that contribute a copy of 21. Furthermore, since any such copy of 21 arising from an extension of $p_i$ is NE or SW of $\rect(p_1,\dots,p_{i-2})$, we conclude that this collection of $k$ copies of 21 forms a copy of $\oplus_k 21$, as required.
\end{proof}

%
%
%
%
%
%
%
%
%
\section{Proof of Theorem~\ref{thm-main}}\label{sec-main}

Let $\pi$ be a permutation, and let $\R$ be an axes-parallel rectangle in the plot of $\pi$. A \emph{sliced copy of 21 that spans $\R$} is a copy of 21 whose rectangular hull is NE or SW of $\R$ and sliced by a point that is NW or SE of $\R$.

For our proof we will need the following observation:
\begin{observation}\label{part21}
For the permutation $\pi=\oplus_L 21$, the only non-empty intervals are singletons or the points in a contiguous sum of $21$s, i.e.\ intervals of the form $\{(i,\pi(i)): i\in [2k-1, 2\ell]\}$ for some $1 \leq k \leq \ell \leq L$.
\end{observation}

\begin{lemma}\label{lem-indec-slice}Let $\pi$ be a sum indecomposable permutation with $|\pi|>1$. Any line slicing $\pi$ must slice a copy of 21.\end{lemma}

\begin{proof}
If $\pi$ were sliced by a line not slicing a copy of 21, then $\pi$ would equal $\pi_1 \oplus \pi_2$ with $\pi_1$ the subpermutation of $\pi$ below (or left of) the line, and $\pi_2$ the subpermutation above (or right of) the line.
\end{proof}

\begin{lemma}\label{simple-slice}In a simple permutation $\sigma$, for any point $p$ of $\sigma$ such that the NE corner of $\sigma$ relative to $p$ contains 21, there exists a copy of 21 in the NE corner of $\sigma$ that is sliced by a point in the NW or SE corner.\end{lemma}

\begin{proof}
Consider $\sigma|_{NE}$, the restriction of $\sigma$ to the region NE of $p$. Since it contains a copy of 21, it is either itself a non-trivial sum indecomposable permutation, or it contains a non-trivial sum indecomposable component. This component must be sliced in $\sigma$ by some other point, which necessarily lies in the NW or SE corner of $\sigma$ relative to $p$, and the result follows by Lemma~\ref{lem-indec-slice}.
\end{proof}

\begin{lemma}\label{lem-sub-simple}
Let $\pi$ be a simple permutation, and let $L,m$ be positive integers. If $\R$ is an axes-parallel rectangle in the plot of $\pi$ which contains a sum of $L(8m^2)^{8m^2}$ copies of $21$, then either $\pi$ contains a sawtooth alternation of length $3m$, or there is a sum of $L$ copies of 21 contained in a simple permutation within $\pi|_\R$. 
\end{lemma}

\begin{proof}
Starting with $\R_0=\R$ and $\Sigma_0$ a set of points in $\R$ forming a sum of $L(8m^2)^{8m^2}$ copies of $21$, we will construct a sequence of rectangles $\R_0\supsetneq \R_1 \supsetneq \cdots \supsetneq \R_k$ with $k\leq 8m^2$ with the following properties, for $i\geq 1$.
\begin{enumerate}
\item[(i)] Each $\R_i$ contains a set of points $\Sigma_i\subsetneq\Sigma_{i-1}$ that forms a sum of at least $L(8m^2)^{8m^2-i}$ copies of $21$;
\item[(ii)] the subpermutation $\pi|_{\R_i}$ is an interval inside $\pi|_{\R_{i-1}}$ (and thus by induction is also an interval inside $\pi|_{\R_0}$). 
\item[(iii)] In $\pi|_{\R_{i-1}}$, there exists a copy of 21 that is NE or SW of $\R_i$, and either forms a sliced copy of 21 spanning $\R_i$, or is sliced by a point outside $\R_0$.
\end{enumerate}

Our construction of rectangles will terminate at $\R_k$ for some $k < 8m^2$ if in $\pi|_{\R_k}$ we can find a copy of a simple permutation that contains a sum of (at least) $L$ copies of $21$ or a sawtooth alternation of length $3m$. 

Otherwise, our construction terminates when $k = 8m^2$, at which point condition (iii) will guarantee that we have a sum of at least $8m^2$ copies of 21 inside $\R_0$ (one for each rectangle). Each copy of 21 is either sliced by a point outside $\R_0$, or it forms a sliced copy of 21 that spans the next rectangle $\R_i$ in the sequence. 

If there are $4m^2$ copies of 21 that are sliced by points outside $\R_0$, then we can find $m^2$ pairs that are sliced on the same side of $\R_0$, and we can apply the Erd\H{o}s--Szekeres Theorem~\cite{erdos:a-combinatorial:} to find a monotone sequence of $m$ points outside $\R_0$ slicing copies of 21, giving a sawtooth alternation of length $3m$.

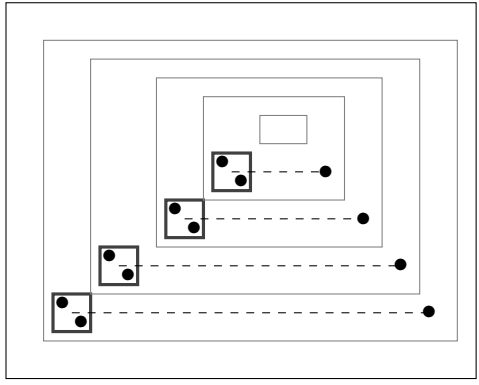
\begin{figure}
\centering
\begin{tikzpicture}[scale=0.25]
\draw (0,0) rectangle (25,20);
\foreach \x/\y/\z in {2/2/19,4.5/4.5/15,8/7/9.5,10.5/9.5/5} {
	\begin{scope}[shift={(\x,\y)}]
	\plotpermborder{2,1}
	\draw[dashed] (1.5,1.5) -- (1.5+\z,1.5);
	\absdot{(1.5+\z,1.5)}{}
	\end{scope}
}
\draw[thin, gray] (2,2) rectangle (24,18);
\draw[thin, gray] (4.5,4.5) rectangle (22,17);
\draw[thin, gray] (8,7) rectangle (20,16);
\draw[thin, gray] (10.5,9.5) rectangle (18,15);
\draw[thin, gray] (13.5,12.5) rectangle (16,14);
\end{tikzpicture}
\caption{A sequence of $m$ sliced copies of 21 that span nested rectangles in the same direction immediately yields a wedge sawtooth alternation of length $3m$.}
\label{fig-triples}
\end{figure}

On the other hand, if there are $4m^2$ copies of 21 sliced inside $\R_0$, then we can find $m^2$ copies of 21 that are sliced by points in the same way (i.e., one of the 21 lying NE or SW of the next rectangle, with the slicing point being NW or SE). Because the sequence of rectangles $\R_0,\R_1,\dots$ are nested, the slicing points already form a monotone sequence (see Figure~\ref{fig-triples}), which means that we have in fact found a wedge sawtooth alternation of length $3m^2$ (and hence one of length $3m$).

Thus, it now suffices to describe the process to construct $\R_{i+1}$ from $\R_i$ to satisfy (i)--(iii) above. Consider the substitution decomposition of $\pi|_{\R_i}$. First, if $\pi|_{\R_i}$ is skew decomposable, then one of the skew indecomposable components must contain all of $\Sigma_i$, so we could restrict $\pi|_{\R_i}$ to this single component, with (i)--(iii) still being satisfied. Thus, without loss of generality, we can assume that $\pi|_{\R_i}$ is skew indecomposable.

Next, suppose that $\pi|_{\R_i}$ is sum decomposable. If $\Sigma_i$ is contained entirely inside a single sum indecomposable component of $\pi|_{\R_i}$, then we may replace $\R_i$ with the rectangular hull of this component, in which case $\pi|_{\R_i}$ is no longer sum decomposable and we have a different case. Otherwise, $\Sigma_i$ is distributed across at least two of the sum indecomposable components of $\pi|_{\R_i}$, and note that each such component $\tau$ intersects $\Sigma_i$ in a whole number of 21s. 

Since $\pi$ is simple, every nonsingleton sum indecomposable component of $\pi|_{\R_i}$ must be sliced by a point outside $\R_i$ (and hence outside $\R_0$), and by Lemma~\ref{lem-indec-slice} this slice must also slice a copy of $21$ inside the component. Thus, if we have at least $4m^2$ nonsingleton components in $\pi|_{\R_i}$, we can find a sum of $4m^2$ copies of 21 sliced by points outside $\R_0$. Since there are four sides of $\R_0$, $m^2$ of these components have their slicing points on the same side of $\R_0$, and the Erd\H{o}s--Szekeres Theorem then yields a wedge or parallel sawtooth alternation of length $3m$. See Figure~\ref{fig-sum-dec}(a).

\begin{figure}
\centering
\begin{tabular}{ccc}
\begin{tikzpicture}[scale=0.25]
\draw (-5,-2) rectangle (23,19);
\foreach \x/\y in {14/12,9/8,7/6,5/4,0/0} {
 \begin{scope}[shift={(\x,\y)}]
  \plotperm{2,1}
  \draw[gray] (0.5,0.5) rectangle (2.5,2.5);
 \end{scope}
}
\foreach \x/\y in {11.5/10,2.5/2} {
 \begin{scope}[shift={(\x,\y)}]
  \plotperm{2,1}
  \draw[gray] (0,0.5) rectangle (3,2.5);
 \end{scope}
}
\draw[thick] (0.5,0.5) rectangle (17.5,15.5);
\foreach \a/\b/\c/\d in {1/1.5/21/1.5,4/4/4/-1,6/5.5/22/5.5,
			9/7.5/-2/7.5,10/9.5/20/9.5,13/11/13/17,15/13.5/19/13.5} {
\draw[dashed] (\a,\b) -- (\c,\d);
\absdot{(\c,\d)}{}
}
\node at (1.6,14.4) {$\R_i$};
\node at (-4,18) {$\pi$};
\end{tikzpicture}
&\rule{10pt}{0pt}&
\begin{tikzpicture}[scale=0.25]
\draw (-5,-2) rectangle (23,19);
\begin{scope}[shift={(2,2)}]
\plotpermborder{2,1,4,3,6,5,8,7}
\end{scope}
\begin{scope}[shift={(13,12)}]
\plotpermborder{2,1}
\end{scope}
\draw[thick] (0.5,0.5) rectangle (17.5,15.5);
\draw[thin, gray] (0.5,0.5) rectangle (2.5,2.5);
\draw[thin, gray] (10.5,10.5) rectangle (13.5,12.5);
\draw[thin, gray] (15.5,14.5) rectangle (17.5,15.5);
\draw[dashed] (14.5,-1) -- (14.5,18);
\draw[dashed] (-2,13.5) -- (21,13.5);
\absdothollow{(14.5,-1)}{}
\absdothollow{(14.5,18)}{}
\absdothollow{(-2,13.5)}{}
\absdothollow{(21,13.5)}{}
\node at (1.6,14.4) {$\R_i$};
\node at (3.5,9.5) {$\tau$};
\node at (13,15) {$\tau'$};
\node at (-4,18) {$\pi$};
\end{tikzpicture}
\\
(a)&&(b)
\end{tabular}
\caption{The two scenarios when $\pi|_{\R_i}$ is sum decomposable: (a) at least $4m^2$ nonsingleton sum components; (b) one component ($\tau$) contains a sum of at least $L(8m^2)^{8m^2-(i+1)}$ copies of $21$ and another ($\tau'$) contains at least one.}
\label{fig-sum-dec}
\end{figure}
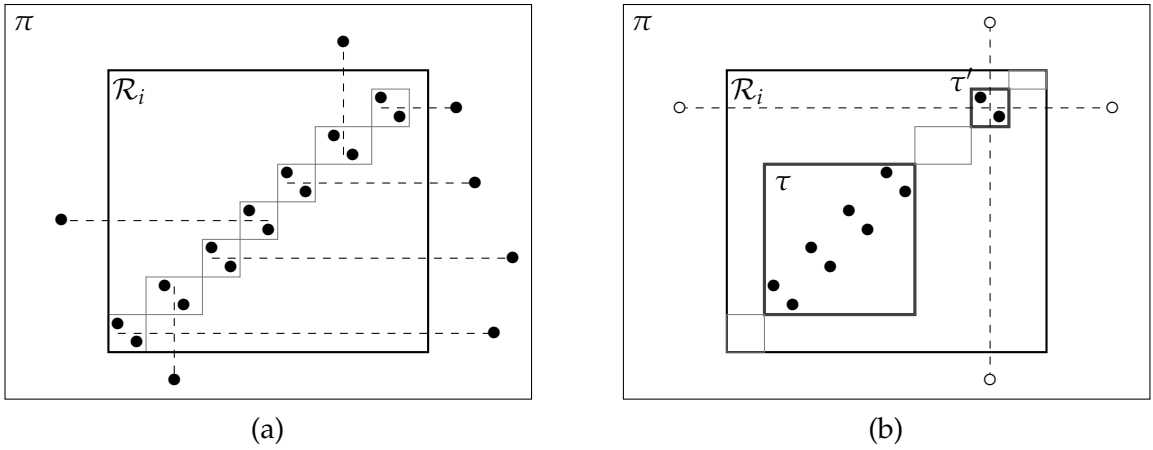

Now we may suppose that we have fewer than $4m^2$ nonsingleton sum indecomposable components in $\pi|_{\R_i}$. By the pigeonhole principle, this means that there is some component $\tau$ which contains at least $L(8m^2)^{8m^2-i}/(4m^2) > L(8m^2)^{8m^2-(i+1)}$ copies of 21. We also know (by assumption) that there exists some other nonsingleton component $\tau'$, which must be sliced by some point outside $\R_0$, see Figure~\ref{fig-sum-dec}(b). We set $\R_{i+1}$ to be the rectangular hull of $\tau$, and any other nonsingleton component $\tau'$ provides the copy of 21 sliced by a point outside $\R_0$.

We now turn to the case where $\pi|_{\R_i}$ is sum and skew indecomposable. In this case, consider the substitution decomposition of $\pi|_{\R_i}$, and write $\pi|_{\R_i} = \sigma[\tau_1,\dots,\tau_{|\sigma|}]$. By Observation~\ref{part21}, each interval $\tau_j$ may contain 0, 1, or a multiple of 2 points from $\Sigma_i$. First, if $\Sigma_i\subset \tau_j$ for some $j$, then we may replace $\R_i$ with the rectangular hull of $\tau_j$, and consider this instead. 

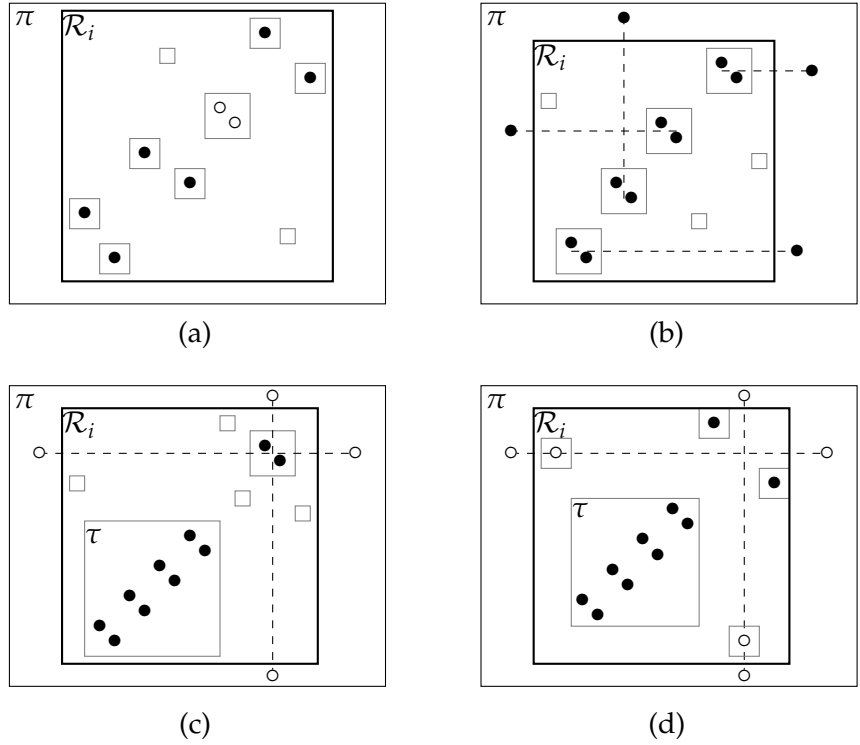
\begin{figure}
\centering
\begin{tabular}{ccc}
\begin{tikzpicture}[scale=0.2]
\draw (-4,-2) rectangle (21,18);
\foreach \x/\y in {0/3,2/0,4/7,7/5,12/15,15/12} {
 \begin{scope}[shift={(\x,\y)}]
  \plotperm{1}
  \draw[gray] (0,0) rectangle (2,2);
 \end{scope}
}
\foreach \x/\y in {5.5/13.5,13.5/1.5} {
 \begin{scope}[shift={(\x,\y)}]
  \draw[gray] (0.5,0.5) rectangle (1.5,1.5);
 \end{scope}
}
\absdothollow{(11,10)}{}
\absdothollow{(10,11)}{}
\draw[gray] (9,9) rectangle (12,12);
\draw[thick] (-0.5,-0.5) rectangle (17.5,17.5);
\node at (0.6,16.4) {$\R_i$};
\node at (-3,17) {$\pi$};
\end{tikzpicture}
&\rule{10pt}{0pt}&
\begin{tikzpicture}[scale=0.2]
\draw (-4,-2) rectangle (21,18);
\foreach \x/\y in {1/0,4/4,7/8,11/12} {
 \begin{scope}[shift={(\x,\y)}]
  \plotperm{2,1}
  \draw[gray] (0,0) rectangle (3,3);
 \end{scope}
}
\foreach \x/\y in {0/11,10/3,14/7} {
 \begin{scope}[shift={(\x,\y)}]
  \draw[gray] (0,0) rectangle (1,1);
 \end{scope}
}
\draw[thick] (-0.5,-0.5) rectangle (15.5,15.5);
\node at (0.6,14.4) {$\R_i$};
\node at (-3,17) {$\pi$};
\foreach \x/\y/\z/\w in {2/1.5/17/1.5, 5.5/5/5.5/17,
					9/9.5/-2/9.5,12/13.5/18/13.5} {
	\draw[dashed] (\x,\y) -- (\z,\w);
	\absdot{(\z,\w)}{}
}
\end{tikzpicture}

\\
(a)&&(b)\\[10pt]
\begin{tikzpicture}[scale=0.2]
\draw (-4,-2) rectangle (21,18);
\begin{scope}[shift={(12,12)}]
  \plotperm{2,1}
  \draw[gray] (0,0) rectangle (3,3);
\end{scope}
\foreach \x/\y in {0/11,10/15,11/10,15/9} {
 \begin{scope}[shift={(\x,\y)}]
  \draw[gray] (0,0) rectangle (1,1);
 \end{scope}
}
\begin{scope}[shift={(1,0)}]
	\plotperm{2,1,4,3,6,5,8,7}
	\draw[gray] (0,0) rectangle (9,9);
	\node at (0.6,8) {$\tau$};
\end{scope}
\draw[thick] (-0.5,-0.5) rectangle (16.5,16.5);
\node at (0.6,15.4) {$\R_i$};
\node at (-3,17) {$\pi$};
\draw[dashed] (-2,13.5) -- (19,13.5);
\draw[dashed] (13.5,-1.3) -- (13.5,17.3);
\absdothollow{(-2,13.5)}{}
\absdothollow{(19,13.5)}{}
\absdothollow{(13.5,-1.3)}{}
\absdothollow{(13.5,17.3)}{}
\end{tikzpicture}
&\rule{10pt}{0pt}&
\begin{tikzpicture}[scale=0.2]
\draw (-4,-2) rectangle (21,18);
\foreach \x/\y in {10.5/14.5,14.5/10.5} {
 \begin{scope}[shift={(\x,\y)}]
  \absdot{(1,1)}{}
  \draw[gray] (0,0) rectangle (2,2);
\end{scope}
}
\begin{scope}[shift={(1.75,1.75)}]
	\plotperm{2,1,4,3,6,5,8,7}
	\draw[gray] (0.25,0.25) rectangle (8.75,8.75);
	\node at (0.85,8) {$\tau$};
\end{scope}
\draw[thick] (-0.5,-0.5) rectangle (16.5,16.5);
\node at (0.6,15.4) {$\R_i$};
\node at (-3,17) {$\pi$};
\draw[dashed] (-2,13.5) -- (19,13.5);
\draw[dashed] (13.5,-1.3) -- (13.5,17.3);
\absdothollow{(-2,13.5)}{}
\absdothollow{(19,13.5)}{}
\absdothollow{(13.5,-1.3)}{}
\absdothollow{(13.5,17.3)}{}
\foreach \x/\y in {0/12.5,12.5/0} {
 \begin{scope}[shift={(\x,\y)}]
  \absdothollow{(1,1)}{}
  \draw[gray] (0,0) rectangle (2,2);
 \end{scope}
}

\end{tikzpicture}
\\
(c)&&(d)
\end{tabular}
\caption{The scenarios when $\pi|_{\R_i}$ is sum and skew indecomposable: (a) $\pi_{\R_i}$ contains $2L$ intervals intersecting $\Sigma_i$ in 1 point; (b) At least $4m^2$ intervals intersect $\Sigma_i$ in two or more points; (c) $\tau$ contains many copies of 21 from $\Sigma_i$ and another interval $\tau'$ contains at least 1; (d) $\tau$ contains many copies of 21, and another copy of 21 is split across two other intervals.}
\label{fig-sigma-dec}
\end{figure}

If for some copy of 21 in $\Sigma_i$ one of the points is a singleton inside some interval $\tau_j$, then there exists another interval that contains the other point as a singleton. Thus, if we can find at least $2L$ intervals of $\pi|_{\R_i}$ each of which intersects $\Sigma_i$ in exactly one point, then $\sigma$ contains a sum of at least $L$ copies of 21, see Figure~\ref{fig-sigma-dec}(a). 
On the other hand, if we can find at least $4m^2$ intervals each of which intersects $\Sigma_i$ in (at least) two points, then each of these intervals must contain a copy of 21 that is sliced by a point outside $\R_0$, and by the earlier argument we can find a sawtooth alternation of length $3m$. See Figure~\ref{fig-sigma-dec}(b).

We may now suppose that fewer than $2L$ intervals intersect in exactly one point (so there are at most $L$ copies of 21 for which this occurs), and fewer than $4m^2$ intersect in two or more points. Since $\Sigma_i$ comprises at least $L(8m^2)^{8m^2-i}$ copies of 21, there are at least $L\left((8m^2)^{8m^2-i}-1\right)$ copies of 21 in which the `2' and `1' of each pair lie in the same interval. By the pigeonhole principle, there exists an interval $\tau$ that contains at least $L\left((8m^2)^{8m^2-i}-1\right)/(4m^2) > L(8m^2)^{8m^2-(i+1)}$ copies of 21 from $\Sigma_i$. We set $\R_{i+1}$ to be the rectangular hull of $\tau$.

Finally, since $\tau$ is not the only interval containing copies of 21 from $\Sigma_i$, we can now either find another interval $\tau'$ that contains a whole number of copies of 21 from $\Sigma_i$, or two intervals that together contain a copy of $21$ from $\Sigma_i$. In the first case (illustrated in Figure~\ref{fig-sigma-dec}(c)), the interval $\tau'$ must contain a copy of 21 that is sliced by a point outside $\R_0$ as required. In the second case (illustrated in Figure~\ref{fig-sigma-dec}(d)), the copy of 21 must be sliced either outside $\R_0$, or by Lemma~\ref{simple-slice} we can find a copy of 21 in the same region of $\R_i$ relative to $\R_{i+1}$ as the original 21 (NE or SW) which is sliced by a point either NW or SE of $\R_{i+1}$ and within $\R_i$. This gives us the necessary sliced copy of 21 that spans $\R_{i+1}$ for condition (iii).

This completes the possible cases for the substitution decomposition of $\pi|_{\R_i}$, and hence the proof. Note that since $\pi$ is a finite permutation, the number of times that we may replace the rectangle $\R_i$ with a smaller rectangle (e.g., when $\pi|_{\R_i}$ is skew decomposable) is bounded. 
\end{proof}

Given a simple permutation $\pi$, let $\rho(\pi)$ denote the sum of the lengths of the maximal sawtooth alternations in $\pi$ of each of the eight types depicted in Figure~\ref{fig-sawtooth}. Our proof of Theorem~\ref{thm-main} will be complete after we have proved the following lemma.

\begin{lemma}\label{lem-big-induction}
For every $m,s \in \mathbb{N}$, there is a function $g(m,s)$ such that every simple permutation $\pi$ that contains a sum of $g(m,s)$ copies of 21 must either contain an increasing oscillation of length $m$ or satisfy $\rho(\pi) \geq 3s$. \end{lemma}

\begin{proof}[Proof of Theorem~\ref{thm-main}, given Lemma~\ref{lem-big-induction}]
We set $f(n) = g(n,8n)$ where $g$ is the function from Lemma~\ref{lem-big-induction}. Thus, any $\pi$ that contains a sum of at least $g(n,8n)$ copies of 21 must contain an increasing oscillation of length $n$, or have $\rho(\pi)\geq 24n$. Since $\rho(\pi)$ is the sum of the sizes of the eight different maximal sawtooth alternations that can be found in $\pi$, one of these sawtooth alternations has length at least $3n$.
\end{proof}

\begin{proof}[Proof of Lemma~\ref{lem-big-induction}]
First, for $m=1,2,3$, and any $s \in \mathbb{N}$, we can set $g(m,s)=2$ (since any simple permutation of length at least four contains an increasing oscillation of length 3). 
Additionally, for \emph{any} fixed $m \geq 4$, it is not hard to see that we may take $g(m,1)=2$ (since any simple permutation of length at least four contains a sawtooth alternation of length three). Thus, we may now assume that $m \geq 4$ and $s > 1$, and we will show that we may take $g(m,s) = (m+3)\cdot k+1$, where $k= (8s^2)^{8s^2} g(m, s-1)$ by induction on $s$.

We start with a simple permutation $\pi$ which contains a sum of at least $(m+3)k+1$ copies of 21, and denote the points in some longest such sum by $\Sigma$. We now partition $\Sigma$ into $m+4$ disjoint rectangles, $\R_0$ through $\R_{m+3}$, where $\R_0$ is the bounding rectangle for the first copy of $21$ in $\Sigma$, and, for $1 \leq i \leq m+3$, $\R_{i}$ is the bounding rectangle for the $k$ least copies of $21$ in $\Sigma$ not contained in any previous $\R_j$. 

If any rectangle $\R_i$ contains a sawtooth alternation of length $3s$, then we are done. Therefore, we may assume that no rectangle contains such a sawtooth alternation, and so in each $\pi|_{\R_i}$ we can find a simple permutation $\sigma_i$ with $g(m,s-1)$ copies of 21 by Lemma~\ref{lem-sub-simple}. For $i=1,\dots, m+3$, let $\S_i=\rect(\sigma_i)$. If any $\sigma_i$ contained an increasing oscillation of length $m$, then we are done, thus by the inductive hypothesis we can assume that $\rho(\sigma_i) = 3s-3$ for $i=1,\dots,m+3$ (since $\rho(\sigma_i) \geq 3s-3$ and if strictly greater we are done).

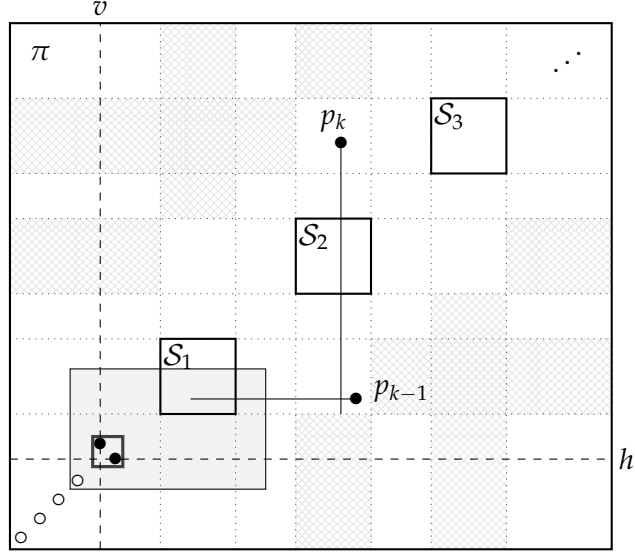
\begin{figure}
\centering
\begin{tikzpicture}[scale=0.2]
\draw[fill=black!5] (-1,-1) rectangle (12,7);
\foreach \x/\y/\z/\w in {14/-5/19/4,23/-5/28/12,19/4/35/9,28/12/35/17,
			-5/12/5/17,-5/20/14/25,5/17/10/30,14/25/19/30} { 
  \draw[pattern=crosshatch,pattern color=black!10,draw=none] (\x,\y) rectangle (\z,\w);
}
\plotpermborder{2,1}
\draw[dashed] (-5,1) -- (35,1);
\draw[dashed] (1,-5) -- (1,30);
\foreach \x/\y [count=\i] in {5/4,14/12,23/20} {
 \draw[dotted,black!70] (-5,\y) -- (35,\y);
 \draw[dotted,black!70] (\x,-5) -- (\x,30);
 \draw[dotted,black!70] (-5,\y+5) -- (35,\y+5);
 \draw[dotted,black!70] (\x+5,-5) -- (\x+5,30);
 \begin{scope}[shift={(\x,\y)}]
  \draw[thick] (0,0) rectangle (5,5);
  \node at (1.2,3.8) {$\S_\i$};
 \end{scope}
}
\foreach \x in {-0.5,-1.75,-3,-4.25} {
	\absdothollow{(\x,\x)}{}
} 
\draw[thick] (-5,-5) rectangle (35,30);
\node at (-3,28) {$\pi$};
\node at (32,28) {$\iddots$};
\node at (1,31) {$v$};
\node at (36,1) {$h$};
\foreach \x/\y/\z/\w in {18/5/7/5,17/22/17/4} {
	\absdot{(\x,\y)}{}
	\draw (\x,\y) -- (\z,\w);
}
\node at (21,5.5) {$p_{k-1}$};
\node at (16.5,23.5) {$p_k$};

\end{tikzpicture}
\caption{The general set-up in the proof of Lemma~\ref{lem-big-induction}. The crosshatched regions must be empty to avoid sliced copies of 21 that span some rectangle $\S_i$. The pin $p_k$ is shown here so as to slice $\S_3$, and the shaded rectangle denotes  $\rect(p_1,\dots,p_{k-2})$.
}
\label{fig-pins-pk-and-pl}
\end{figure}

The next five paragraphs are best read in conjunction with Figure~\ref{fig-pins-pk-and-pl}. Let $h$ denote the horizontal line crossing through the `1' of the initial 21, and $v$ the vertical line crossing through the `2'. Note that the bottom-left corner of $\pi$ below $h$ and to the left of $v$ is increasing, else we would be able to find a longer sum of 21s in $\pi$ than $\Sigma$. For convenience, we will refer to the L-shaped region below $h$ and/or to the left of $v$ as the \emph{outside} region of $\pi$, and the rest of $\pi$ will be \emph{inside}. 

Recall that a sliced copy of 21 that spans $\S_j$ is a copy of 21 that is NE or SW of $\S_j$, sliced by a point that is NW or SE of $\S_j$. If we can find a sliced copy of 21 that spans $\S_j$, then we may append it to one of the eight types of sawtooth alternation in $\S_j$. Since $\rho(\sigma_j)=3s-3$, this implies that $\rho(\pi)\geq 3s$ and we are done. Consequently, from now on we will assume that there are no sliced copies of 21 spanning any rectangle $\S_j$.

Under this assumption, since any point $p$ that slices some $\S_i$ must (by Lemma~\ref{lem-indec-slice}) slice a copy of 21, all such slicing points must be below and to the left of the top-right corner of $S_{i+1}$, and above and to the right of the bottom-left corner of $\S_{i-1}$ (when these rectangles exist). This implies that a number of regions defined by the four boundary lines of each $\S_i$ must be empty, as identified by the crosshatched areas in Figure~\ref{fig-pins-pk-and-pl}.

Now consider a shortest right-reaching pin sequence starting from the initial 21 of $\Sigma$. We will denote this pin sequence by $p_1,p_2,\dots,p_n$. For any initial segment $p_1,\dots,p_j$, let $i_j$ denote the least index (if it exists) such that $\S_{i_j}$ is contained in the NE region of $\rect(p_1,\dots,p_j)$. Observe that $i_1=i_2 = 1$. For any $j$ satisfying $2\leq j<n$, the pin $p_{j+1}$ slices the rectangular hull $\rect(p_1,\dots,p_j)$ in such a way as to slice a copy of $21$. Thus, whenever $p_{j+1}$ is a right pin or an up pin, we can assume that $p_{j+1}$ does not extend beyond $\S_{i_j}$. From this, we make two conclusions: first, that $i_{j+1} \leq i_j+1$, and second, that \emph{every} $\S_i$ must be sliced by some pin. Note that if $p_{j+1}$ is a down or left pin, then $i_{j+1}=i_j$.

In this pin sequence, we identify pin $p_k$, which is the first pin such that $\S_1\subset \rect(p_1,\dots,p_k)$. Clearly $p_k$ must be an up pin or a right pin, and we will assume that it is an up pin, the other case being analogous. We claim that $i_k \leq 4$. If not, then some $p_\ell$ (with $\ell< k$) slices $\S_{3}$. Since $p_k$ is the earliest up pin that extends at least as far as the top of $\S_1$, we conclude that $p_\ell$ must be a right pin, and can be no higher than the top of $\S_1$. However, any such pin must then contribute to a sliced copy of 21 that spans $\S_2$. Thus $i_k\leq 4$, and note that this bound is tight, as illustrated by the example placement of $p_k$ and $p_{k-1}$ in Figure~\ref{fig-pins-pk-and-pl}.  Note further that, in any case, none of the pins $p_{k+2},\dots,p_n$ can slice $\S_1$.

We now classify more precisely which $j>k$ can satisfy $i_j < i_{j+1}$. By the earlier comments, this can only happen when $p_{j+1}$ is an up or right pin, and then $i_{j+1}\leq i_j+1$. We claim that $i_j = i_{j+1}$ unless both $p_j$ and $p_{j+1}$ lie in the inside region.

First, suppose $p_{j+1}$ is a right or up pin in the outside region, then it cannot slice $\S_{i_j}$ (else $p_{j+1}$ slices a copy of 21 in $\S_{i_j}$, and this sliced copy spans $\S_1$, see Figure~\ref{fig-pj-outside}(a)), and it cannot extend beyond $\S_{i_j}$ (else we can find a copy of 21 in $\{p_1,\dots, p_j\}$ sliced by $p_{j+1}$ that spans $\S_{i_j}$, see Figure~\ref{fig-pj-outside}(b)). Thus, we conclude that $\S_{i_j}$ is NE of $\rect(p_1,\dots,p_{j+1})$, i.e.\ $i_{j+1} = i_j$. 

Next, suppose $p_{j+1}$ is a right or up pin in the inside region, but $p_j$ is outside. By definition, $p_j$ cannot slice $S_{i_j}$, from which we conclude that $p_{j+1}$ cannot be  contained in $S_{i_j}$. Furthermore, $p_{j+1}$ cannot slice $S_{i_j}$, else we may take a point in $S_{i_j}$ together with $p_j$ and $p_{j+1}$, and form a sliced copy of 21 that spans $S_1$, see Figure~\ref{fig-pj-outside}(c). This completes the claim.

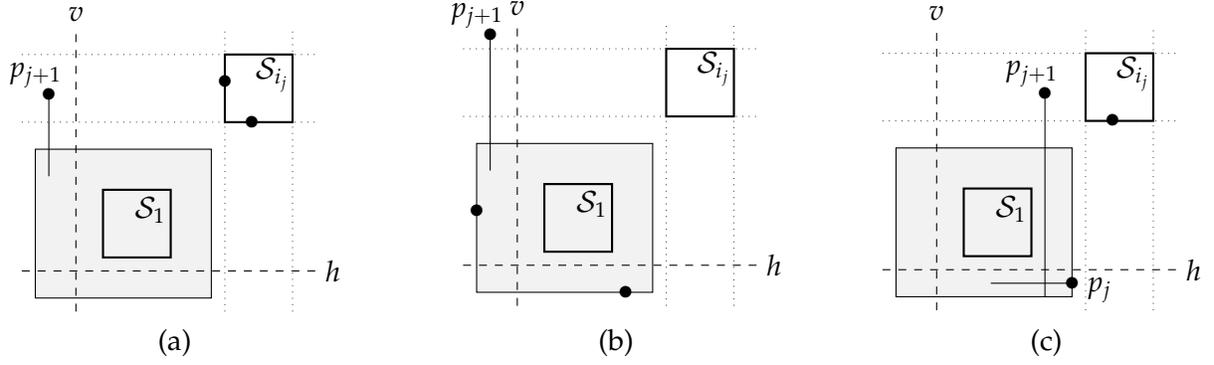
\begin{figure}
\centering
\begin{tabular}{ccccc}
\begin{tikzpicture}[scale=0.18]
\draw[fill=black!5] (-2,-1) rectangle (11,10);
\draw[dashed] (-3,1) -- (19,1);
\draw[dashed] (1,-2) -- (1,19);
 \begin{scope}[shift={(3,2)}]
  \draw[thick] (0,0) rectangle (5,5);
  \node at (3.5,3.5) {$\S_1$};
 \end{scope}
\foreach \x/\y in {12/12} {
 \draw[dotted,black!70] (-3,\y) -- (19,\y);
 \draw[dotted,black!70] (\x,-2) -- (\x,19);
 \draw[dotted,black!70] (-3,\y+5) -- (19,\y+5);
 \draw[dotted,black!70] (\x+5,-2) -- (\x+5,19);
 \begin{scope}[shift={(\x,\y)}]
  \draw[thick] (0,0) rectangle (5,5);
  \node at (3.5,3.5) {$\S_{i_j}$};
 \end{scope}
}
\node at (1,20) {$v$};
\node at (20,1) {$h$};
\foreach \x/\y/\z/\w in {-1/14/-1/8} {
	\absdot{(\x,\y)}{}
	\draw (\x,\y) -- (\z,\w);
}
\node at (-2,15.5) {$p_{j+1}$};
\absdot{(14,12)}{}
\absdot{(12,15)}{}
\end{tikzpicture}
&\rule{10pt}{0pt}&
\begin{tikzpicture}[scale=0.18]
\draw[fill=black!5] (-2,-1) rectangle (11,10);
\draw[dashed] (-3,1) -- (19,1);
\draw[dashed] (1,-2) -- (1,19);
 \begin{scope}[shift={(3,2)}]
  \draw[thick] (0,0) rectangle (5,5);
  \node at (3.5,3.5) {$\S_1$};
 \end{scope}
\foreach \x/\y in {12/12} {
 \draw[dotted,black!70] (-3,\y) -- (19,\y);
 \draw[dotted,black!70] (\x,-2) -- (\x,19);
 \draw[dotted,black!70] (-3,\y+5) -- (19,\y+5);
 \draw[dotted,black!70] (\x+5,-2) -- (\x+5,19);
 \begin{scope}[shift={(\x,\y)}]
  \draw[thick] (0,0) rectangle (5,5);
  \node at (3.5,3.5) {$\S_{i_j}$};
 \end{scope}
}
\node at (1,20) {$v$};
\node at (20,1) {$h$};
\foreach \x/\y/\z/\w in {-1/18/-1/8} {
	\absdot{(\x,\y)}{}
	\draw (\x,\y) -- (\z,\w);
}
\node at (-2,19.5) {$p_{j+1}$};
\absdot{(-2,5)}{}
\absdot{(9,-1)}{}
\end{tikzpicture}
&\rule{10pt}{0pt}&
\begin{tikzpicture}[scale=0.18]
\draw[fill=black!5] (-2,-1) rectangle (11,10);
\draw[dashed] (-3,1) -- (19,1);
\draw[dashed] (1,-2) -- (1,19);
 \begin{scope}[shift={(3,2)}]
  \draw[thick] (0,0) rectangle (5,5);
  \node at (3.5,3.5) {$\S_1$};
 \end{scope}
\foreach \x/\y in {12/12} {
 \draw[dotted,black!70] (-3,\y) -- (19,\y);
 \draw[dotted,black!70] (\x,-2) -- (\x,19);
 \draw[dotted,black!70] (-3,\y+5) -- (19,\y+5);
 \draw[dotted,black!70] (\x+5,-2) -- (\x+5,19);
 \begin{scope}[shift={(\x,\y)}]
  \draw[thick] (0,0) rectangle (5,5);
  \node at (3.5,3.5) {$\S_{i_j}$};
 \end{scope}
}
\node at (1,20) {$v$};
\node at (20,1) {$h$};
\foreach \x/\y/\z/\w in {9/14/9/-1,11/0/5/0} {
	\absdot{(\x,\y)}{}
	\draw (\x,\y) -- (\z,\w);
}
\node at (13,-0.5) {$p_{j}$};
\node at (8,15.5) {$p_{j+1}$};
\absdot{(14,12)}{}
\end{tikzpicture}\\
(a)&&(b)&&(c)
\end{tabular}
\caption{Up to symmetry, the three situations where $p_{j+1}$ is an up pin but $p_j$ and $p_{j+1}$ are not both in the inside region, all give rise to sliced copies of 21 that span $\S_1$ or $\S_{i_j}$. The shaded region denotes $\rect(p_1,\dots,p_j)$. The case where $p_{j+1}$ is a right pin is analogous.}
\label{fig-pj-outside}
\end{figure}

We now identify the least index $k'>k+1$ such that $i_{k'+1}=i_{k'}+1$, and note that $i_{k'}\leq 6$ (since $i_k\leq 4$) and none of $p_{k'},\dots,p_n$ slices $\S_1$. By the above argument, the sequence of pins $p_{k'-1},p_{k'},p_{k'+1}$ must be `up-right-up' or `right-up-right'. Now consider $p_{k'+2}$: if it is a down or left pin, then $\{p_{k'},p_{k'+1},p_{k'+2}\}$ forms a sliced copy of 21 that spans $\S_1$. Thus, $p_{k'+2}$ must also be an up or right pin, and the same argument applies to all subsequent pins. Thus, the sequence of points $p_{k'-1},p_{k'},p_{k'+1},\dots,p_n$ defines an increasing oscillation. Furthermore, since $i_{k'}\leq 6$ and each subsequent pin can slice at most one more rectangle than its immediate predecessor, since there are $m+3$ rectangles in total, this oscillation contains at least $m$ points, completing the proof.%
\end{proof}

%
%
%
%
%
%
%
%
\section{Monotone griddability for simple permutations}\label{sec-gridding}

In this section, we complete our proof of Theorem~\ref{thm:griddable}.  Starting from Theorem~\ref{thm-main}, our main concern is handling wedge sawtooth alternations since they are not simple. Our key result is the following, whose proof will take up the majority of this section.

\begin{proposition}\label{prop:wedge-sawtooth}
For every $m, p, s \in \mathbb{N}$ there exists a number $h(m, p, s)\in \mathbb{N}$ such that whenever a simple permutation contains a wedge sawtooth alternation of length $3h(m,p,s)$, then it contains a split wedge sawtooth alternation of length $3m$, a proper pin sequence with at least $p$ turns, or a spiral pin sequence with at least $s$ extensions.
\end{proposition}

\begin{proof}[Proof of Theorem~\ref{thm:griddable}, given Proposition~\ref{prop:wedge-sawtooth}]
First, if $\C$ contains arbitrarily long copies of any of the structures listed, then $\C$ contains arbitrarily long simple permutations that themselves contain arbitrarily long sums of 21 or skew sums of 12. Thus, by Theorem~\ref{thm-hv-griddable} the simple permutations in $\C$ cannot be monotone griddable.

Conversely, let $\Si(\C)$ denote the set of simple permutations in $\C$, and 
suppose that the permutations in $\Si(\C)$ are not monotone griddable. By Theorem~\ref{thm-hv-griddable}, the permutation class formed by taking the closure of the set $\Si(\C)$ must contain arbitrarily long sums of 21 or skew sums of 12, and hence there must be simple permutations in $\C$ that contain arbitrarily long sums of 21 or skew sums of 12. If $\C$ contains arbitrarily long sums of 21, then by Theorem~\ref{thm-main} the simple permutations of $\C$ must contain arbitrarily long sawtooth alternations or increasing oscillations. We are done unless $\Si(\C)$ contains only long wedge sawtooth alternations, but in this case we may apply Proposition~\ref{prop:wedge-sawtooth} to conclude that $\C$ contains arbitrarily long split wedge sawtooth alternations, proper pin sequences with arbitrarily many turns, or spiral proper pin sequences with arbitrarily many extensions. In any case, we conclude that we have found one of the forbidden substructures specified in Theorem~\ref{thm:griddable}.

A symmetric argument applies in the case when $\C$ contains arbitrarily long skew-sums of 12.
\end{proof}

\begin{proof}[Proof of Proposition~\ref{prop:wedge-sawtooth}] 

We will show that $h(m,p,s)=3mp(2s+1)$ suffices. Suppose we have a simple permutation $\pi$ that contains a wedge sawtooth alternation $\omega$ of size $9mp(2s+1)$. By symmetry, we can assume that $\omega$ comprises a sum of 21s which are split below, oriented $<$. From left to right, denote the triples of points of $\omega$ that form the sliced copies of 21 by $T_1, T_2, \ldots, T_{3mp(2s+1)}$. 

Now, consider a shortest right-reaching pin sequence $p_1p_2 \ldots p_n$ in $\pi$ starting at the leftmost sliced copy of 21 of $\omega$. For $i \in \{1,2, \ldots, n\}$, let $t(i)$ be the largest index $j$ such that at least one point of $T_j$ slices or is contained in $\rect(p_1,\dots,p_i)$, and let $\Delta(i)=t(i)-t(i-1)$ for $i>1$. Note that $t(1)=0$, $t(2)=1$, $t(n)=3mp(2s+1)$ and $\sum_{i=2}^{n}\Delta(i) = t(0)+t(n) = 3mp(2s+1)$. We are interested in the pins for which $\Delta(i)>0$. For any $p_i$ with $\Delta(i)>0$, note that this implies there is at least one point from each of $\Delta(i)$ triples $T_j$ that slice $\rect(p_1,\dots,p_i)\setminus \rect(p_1,\dots,p_{i-1})$.

First, if there exists $i$ such that $\Delta(i) \geq m$ then the pin $p_{i}$ slices a copy of 21 in $\rect(p_1,\dots,p_{i-1})$, which together with  $m-1$ sliced copies of 21 from $\omega$, $T_{t(i)+1},\dots,T_{t(i)+m-1}$, forms a split wedge sawtooth alternation of length $3m$, of type 1 if $p_i$ is a down pin, type 2 if $p_i$ is a right pin, and type 3 if $p_i$ is an up pin. Consequently, we can now assume that $\Delta(i) < m$ for all $i$. Letting $M=\{i: \Delta(i)>0\}$ denote the set of indices $i$ for which $\Delta(i)$ is non-zero, we have $|M|\geq 3p(2s+1)$.

Next, we are done if the pin sequence contains at least $p$ turns, thus we will assume that there are fewer than $p$ turns in total. Since $|M|\geq 3p(2s+1)$, there exists a turn-free factor of $p_1,\dots,p_n$ containing at least $3(2s+1)$ distinct indices $i$ for which $\Delta(i)>0$. Let this factor be $p_k,\dots,p_\ell$, which we will assume forms a clockwise spiral pin sequence (i.e.\ the directions follow the order up, right, down, left).

We will assume that $p_k$ is a down pin, otherwise we may remove at most three pins from the beginning of $p_k,\dots,p_\ell$ to recover a spiral sequence beginning with a down pin. The cost of doing this is that $p_k,\dots,p_\ell$ is now only guaranteed to contain at least $3(2s+1)-3 = 6s$ distinct indices $i$ with $\Delta(i)>0$. Irrespective of this, we have $k\geq 3$ (since the first two pins have no direction), which means that $p_k$ extends from a non-trivial $\rect(p_1,\dots,p_{k-1})$, and therefore $p_k$ slices a copy of 21. 

\begin{figure}
\centering
\begin{tabular}{ccccc}
\tikzstyle{background grid}=[draw, black!20,step=0.4]
\begin{tikzpicture}[scale=0.4,show background grid]
\foreach \x/\y/\z/\w in {-3/-2/-2/3,0/-2/4/-4,6/-2/7/5,4/0/6/3, -2/5/4/6,0/3/2/5,-2/-5/6/-4,-2/0/0/1} 
	\draw[draw=none, fill=black!40] (\x,\y) rectangle (\z,\w); 
\foreach \x/\y/\z/\w in {5/0/6/-2} 
	\draw[draw=none, fill=black!20] (\x,\y) rectangle (\z,\w); 
\foreach \x/\y/\z/\w in {4/0/5/-2} 
\draw[pattern=crosshatch,pattern color=black!20,draw=none] (\x,\y) rectangle (\z,\w);
\draw[very thick] (0,0) rectangle (2,1);
\draw[very thick] (0,-2) rectangle (4,3);
\foreach \x/\y/\z/\w in {1/0.5/1/3,0.5/2/4/2,3/2.5/3/-2,3.5/-1/-2/-1, -1/-1.5/-1/5,-1.5/4/6/4,5/4.5/5/-4} {
  \draw (\x,\y) -- (\z,\w);
  \absdot{(\z,\w)}{}
}
\node at(4.5,-0.5) {\small $A$};
\node at(4.5,-1.5) {\small $B$};
\node at (3,-2.8) {$p_i$};
\end{tikzpicture}
&&
\tikzstyle{background grid}=[draw, black!20,step=0.4]
\begin{tikzpicture}[scale=0.4,show background grid]
\foreach \x/\y/\z/\w in {-4/0/-2/2,-5/-2/-4/4,-2/-3/3/-2,0/-2/1/0, 5/-2/6/6,3/0/5/4, -4/6/3/7,-2/4/1/6} 
	\draw[draw=none, fill=black!40] (\x,\y) rectangle (\z,\w); 
\foreach \x/\y/\z/\w in {1/5/3/6} 
	\draw[draw=none, fill=black!20] (\x,\y) rectangle (\z,\w); 
\foreach \x/\y/\z/\w in {1/4/3/5} 
\draw[pattern=crosshatch,pattern color=black!20,draw=none] (\x,\y) rectangle (\z,\w);
\draw[very thick] (0,0) rectangle (1,2);
\draw[very thick] (-2,0) rectangle (3,4);
\foreach \x/\y/\z/\w in {0.5/1/-2/1,-1/0.5/-1/4,-1.5/3/3/3,2/3.5/2/-2,2.5/-1/-4/-1,-3/-1.5/-3/6,-3.5/5/5/5} {
  \draw (\x,\y) -- (\z,\w);
  \absdot{(\z,\w)}{}
}
\node at(1.5,4.5) {\small $A$};
\node at(2.5,4.5) {\small $B$};
\node at (3.7,3) {$p_i$};
\end{tikzpicture}
&&
\tikzstyle{background grid}=[draw, black!20,step=0.4]
\begin{tikzpicture}[scale=0.4,show background grid]
\foreach \x/\y/\z/\w in {-4/-2/-2/1,-5/-4/-4/3,0/-4/2/-2,-2/-5/4/-4, -2/3/2/5,-4/5/4/6,2/0/4/1,4/-2/6/3,6/-4/7/5} 
	\draw[draw=none, fill=black!40] (\x,\y) rectangle (\z,\w); 
\foreach \x/\y/\z/\w in {3/-2/4/0,2/4/4/5} 
	\draw[draw=none, fill=black!20] (\x,\y) rectangle (\z,\w); 
\foreach \x/\y/\z/\w in {2/2/3/3,2/-2/3/0,3/-3/4/-2,2/3/4/4} 
\draw[pattern=crosshatch,pattern color=black!20,draw=none] (\x,\y) rectangle (\z,\w);
\draw[very thick] (0,0) rectangle (2,1);
\draw[very thick] (-2,-2) rectangle (2,3);
\foreach \x/\y/\z/\w in {1/0.5/1/-2,1.5/-1/-2/-1,-1/-1.5/-1/3,-1.5/2/4/2,3/2.5/3/-4,3.5/-3/-4/-3,-3/-4.5/-3/5,-3.5/4/6/4} {
  \draw (\x,\y) -- (\z,\w);
  \absdot{(\z,\w)}{}
}
\node at(3.5,2.5) {\small $A$};
\node at(3.5,1.5) {\small $B$};
\node at(2.5,1.5) {\small $C$};
\node at(3.5,-3.5) {\small $D$};
\node at(2.5,-3.5) {\small $E$};
\node at(2.5,-2.5) {\small $F$};
\node at (-1,3.6) {$p_i$};
\end{tikzpicture}\\
(a)&&(b)&&(c)
\end{tabular}
\caption{Identifying points of $\omega$ that slice $\rect(p_1,\dots,p_i)\setminus\rect(p_1,\dots,p_{i-1})$ when (a) $p_i$ is a down pin, (b) $p_i$ is a right pin, (c) $p_i$ is an up pin. The shading is as follows: extremality of pins (dark grey), shortcuts (light grey) and regions where Type 1 extensions can occur (crosshatched).}
\label{fig:up-right-down}
\end{figure}
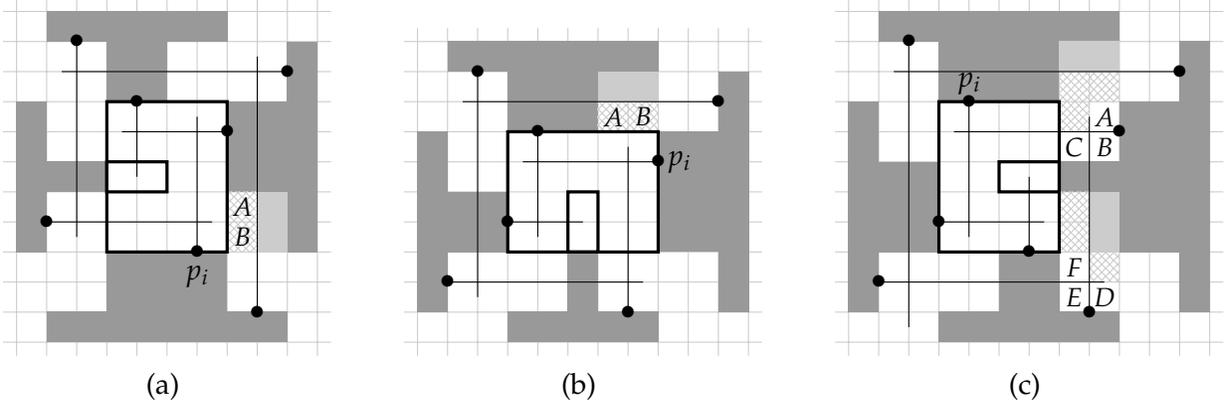

Our discussion is now accompanied by Figure~\ref{fig:up-right-down}. In these diagrams, the dark grey regions contain no points because of the maximality of pins, and the light grey regions contain no points because any such point would enable us to take a `shortcut' in the pin sequence, contradicting our choice of a shortest right-reaching pin sequence. There are also a number of crosshatched areas: these denote regions where the existence of a point will contribute a Type 1 extension.

We now consider any pin $p_i$ (with $k\leq i\leq \ell-4$) for which $\Delta(i)>0$, with a view to identifying a Type 1 or Type 2 extension in each case. Note that if $p_i$ is a left pin, then $\Delta(i)=0$ since $\rect(p_1,\dots,p_{i-1})$ already slices or contains the leftmost sliced copy of 21 in $\omega$. Thus the cases that remain are where $p_i$ is an up, right or down pin. 

If $p_i$ is a down pin, then there must be at least one point of $\omega$ in a crosshatched region (regions $A$ and $B$ in Figure~\ref{fig:up-right-down}(a)), allowing a Type 1 extension to $p_{i}$. Similarly, if $p_i$ is a right pin, then we again conclude that we can find a Type 1 extension to $p_i$, 
since there must exist a point of $\omega$ in regions $A$ or $B$ of Figure~\ref{fig:up-right-down}(b). 

This leaves the case where $p_i$ is an up pin, illustrated in Figure~\ref{fig:up-right-down}(c). If any one of the crosshatched regions contains a point, then we can find a Type 1 extension for one of $p_{i-3}$, $p_{i+1}$ or $p_{i+2}$, so we now assume that these are empty.  Fix some triple $T_j$ which contains a point that slices $\rect(p_1,\dots,p_i)\setminus \rect(p_1,\dots,p_{i-1})$. The `1' of this triple must either coincide with $p_{i+1}$, or lie in one of the regions $A$, $B$ or $C$. This implies the same of the `2', and thence the slicing point must equal $p_{i+2}$, or lie in one of the regions $D$, $E$ or $F$. 

If both the `2' and the `1' of $T_j$ lie in $C$ and the slicing point lies in $F$, then we have a Type 2 extension of $p_{i+1}$ and we are done. There are two other cases to consider: either the slicing point is in region $F$ and the `1' is in region $B$, or the slicing point is below the pin $p_{i+3}$ (i.e.\ it equals $p_{i+2}$ or lies in regions $D$ or $E$). In either case, we have that the `2' and the `1' are sliced by a point that lies below $p_{i+3}$. We can now substitute the pin $p_{i+1}$ with the `1', and the pin $p_{i+2}$ with the slicing point of $T_j$, and then the 2 is a Type 1 extension of the `1' (acting as a right pin).\footnote{In making these pin substitutions, we observe that the resulting sequence violates the extremality condition for proper pin sequences, but it still follows the pattern of a spiral pin sequence, with extensions. Furthermore, we find that $\rect(p_1,\dots, p_{i+6})$ is unchanged, so all future pins will be unaffected by this.}

We have now shown how to find an extension whenever we have a pin $p_i$ with $\Delta(i)>0$. There are at least $6s$ of these pins, but the above analysis does not guarantee that the $6s$ extensions are applied to distinct pins, and extremality may have been (temporarily) violated. 

To resolve these issues, define a \emph{spiral} to be a set of four contiguous pins that begins with a down pin. We observe that the above analysis shows us that for any pin $p_i$ with $\Delta(i)>0$, the pin(s) which can be extended (or substituted and extended) all lie in the same spiral as $p_i$, or the spiral immediately after $p_i$. Thus, by restricting our collection of pins $p_i$ with $\Delta(i)>0$ in $p_k,\dots,p_{\ell}$ to a subset for which any pair is separated by a complete spiral of pins, we can ensure that the extensions are applied to distinct pins. Note also that this corrects any issues arising from the violation of extremality. 

In order to do this, recall that every left pin $p_i$ satisfies $\Delta(i)=0$. Thus, we may choose every sixth pin from the collection of $6s$ pins with $\Delta(i)>0$, leaving us with a set of at least $s$ pins separated by at least seven points. This, in turn, gives us a set of $s$ distinct pins which have extensions of types 1 or 2, and thus we have formed an extended spiral with $s$ extensions.

Finally, while we cannot appeal to symmetry to cover the case where the pin sequence spirals in the opposite direction, the arguments are similar and so we omit the details.
\end{proof}
%
%
%
%
%
%
%
%
\section{Concluding remarks}\label{sec-conclusion}

\paragraph{Decision procedure} In this paper we have characterised the classes whose simple permutations are monotone griddable. From this, it should be possible to describe a decision procedure to answer the follow algorithmic problem:
\begin{question}Given a finitely based permutation class $\C$, is it decidable whether the simple permutations in $\C$ are monotone griddable?\end{question}

The crux of such an algorithm would likely be to extend existing algorithms  that handle pin sequences (such as those given in~\cite{bassino:a-polynomial-al:b} and~\cite{brignall:simple-permutat:b}) to identify turns and (for spiral pin sequences) extensions.

\paragraph{Geometric griddability} For a class $\C$ whose simple permutations are all \emph{geometrically} griddable, the class itself is contained in the substitution closure $\langle\Si(\C)\rangle$, whence we can conclude that $\C$ has a number of important properties: it is finitely based, well-quasi-ordered,  and is enumerated by an algebraic generating function (see Theorems~4.4 and~6.1 of~\cite{albert:inflations-of-g:}). For this reason, a geometric analogue to Theorem~\ref{thm:griddable} is highly desirable.

\begin{question}
Does there exist a characterisation of classes whose simple permutations are \emph{geometrically} griddable?
\end{question}

This would appear to be a difficult question. In particular, there is no known analogue of Theorem~\ref{thm-hv-griddable} to characterise when a class is itself geometrically griddable.

\bibliographystyle{acm}
\bibliography{refs}

\def\cprime{$'$}
\begin{thebibliography}{10}

\bibitem{albert:simple-permutat:}
{\sc Albert, M.~H., and Atkinson, M.~D.}
\newblock Simple permutations and pattern restricted permutations.
\newblock {\em Discrete Math. 300}, 1-3 (2005), 1--15.

\bibitem{albert:geometric-grid-:}
{\sc Albert, M.~H., Atkinson, M.~D., Bouvel, M., Ru{\v{s}}kuc, N., and Vatter,
  V.}
\newblock Geometric grid classes of permutations.
\newblock {\em Trans. Amer. Math. Soc. 365}, 11 (2013), 5859--5881.

\bibitem{aab:three-enumerations:}
{\sc Albert, M.~H., Atkinson, M.~D., and Brignall, R.}
\newblock The enumeration of three pattern classes using monotone grid classes.
\newblock {\em Electron. J. Combin. 19(3)}, 3 (2012), Paper 20, 34.

\bibitem{albert:deflatability}
{\sc Albert, M.~H., Atkinson, M.~D., Homberger, C., and Pantone, J.}
\newblock Deflatability of permutation classes.
\newblock {\em Australas. J. Combin. 64\/} (2016), 252--276.

\bibitem{albert:inflations-of-g:2x4}
{\sc Albert, M.~H., Atkinson, M.~D., and Vatter, V.}
\newblock Inflations of geometric grid classes: three case studies.
\newblock {\em Australas. J. Combin. 58\/} (2014), 24--47.

\bibitem{albert:inflations-of-g:}
{\sc Albert, M.~H., Ru{\v{s}}kuc, N., and Vatter, V.}
\newblock Inflations of geometric grid classes of permutations.
\newblock {\em Israel J. Math. 205}, 1 (2015), 73--108.

\bibitem{bassino:a-polynomial-al:b}
{\sc Bassino, F., Bouvel, M., Pierrot, A., and Rossin, D.}
\newblock An algorithm for deciding the finiteness of the number of simple
  permutations in permutation classes.
\newblock {\em Adv. in Appl. Math. 64\/} (2015), 124--200.

\bibitem{bevan2015defs}
{\sc Bevan, D.~I.}
\newblock Permutation patterns: basic definitions and notation.
\newblock arXiv:1506.06673, 2015.

\bibitem{brignall:decomposing-sim:}
{\sc Brignall, R., Huczynska, S., and Vatter, V.}
\newblock Decomposing simple permutations, with enumerative consequences.
\newblock {\em Combinatorica 28\/} (2008), 385--400.

\bibitem{brignall:simple-permutat:b}
{\sc Brignall, R., Ru\v{s}kuc, N., and Vatter, V.}
\newblock Simple permutations: decidability and unavoidable substructures.
\newblock {\em Theoret. Comput. Sci. 391}, 1--2 (2008), 150--163.

\bibitem{erdos:a-combinatorial:}
{\sc Erd{\H{o}}s, P., and Szekeres, G.}
\newblock A combinatorial problem in geometry.
\newblock {\em Compos. Math. 2\/} (1935), 463--470.

\bibitem{huczynska:grid-classes-an:}
{\sc Huczynska, S., and Vatter, V.}
\newblock Grid classes and the {F}ibonacci dichotomy for restricted
  permutations.
\newblock {\em Electron. J. Combin. 13\/} (2006), Paper 54, 14.

\bibitem{pantone:the-enumeration}
{\sc Pantone, J.}
\newblock The enumeration of permutations avoiding 3124 and 4312.
\newblock {\em Ann. Comb. 21}, 2 (Jun 2017), 293--315.

\bibitem{vatter:survey}
{\sc Vatter, V.}
\newblock Permutation classes.
\newblock In {\em Handbook of Enumerative Combinatorics}, M.~B{\'o}na, Ed. CRC
  Press, Boca Raton, Florida, 2015, pp.~754--833.

\end{thebibliography}

\end{document}